\documentclass[a4paper,12pt]{amsart}
\usepackage[active]{srcltx}

\usepackage{graphicx}
\usepackage{amsmath}
\usepackage{amssymb}
\usepackage{hyperref}
\usepackage{bbm}
\usepackage{enumitem}

\usepackage[total={170mm, 239mm}, includehead, includefoot, centering]{geometry}

\newtheorem{thm}{Theorem}%[section]
\newtheorem{theorem}[thm]{Theorem}%[section]
\newtheorem{lem}[thm]{Lemma}%
\newtheorem{lemma}[thm]{Lemma}
\newtheorem{cor}[thm]{Corollary}%
\theoremstyle{remark}
\newtheorem{remark}{Remark}[section] %

\theoremstyle{plain}

\numberwithin{equation}{section}

\def\GG{{\mathbb G}}
\def\HH{{\mathbb H}}
\def\MM{{\mathbb M}}
\def\NN{{\mathbb N}}

\def\RR{{\mathbb R}}
\def\SS{{\mathbb S}}
\def\TT{{\mathbb T}}

\def\ZZ{{\mathbb Z}}

\def\veca{{\text{\boldmath$a$}}}
\def\vecb{{\text{\boldmath$b$}}}

\def\vecc{{\text{\boldmath$c$}}}
\def\vecd{{\text{\boldmath$d$}}}

\def\vecm{{\text{\boldmath$m$}}}

\def\vecv{{\text{\boldmath$v$}}}

\def\vecx{{\text{\boldmath$x$}}}

\def\bs{\backslash}

\def\scrA{{\mathcal A}}
\def\scrB{{\mathcal B}}
\def\scrC{{\mathcal C}}
\def\scrD{{\mathcal D}}
\def\scrE{{\mathcal E}}

\def\scrL{{\mathcal L}}

\def\scrN{{\mathcal N}}

\def\scrP{{\mathcal P}}

\def\scrU{{\mathcal U}}
\def\scrV{{\mathcal V}}

\def\scrZ{{\mathcal Z}}

\def\Re{\operatorname{Re}}
\def\Im{\operatorname{Im}}

\def\diam{\operatorname{diam}}

\def\H{{\mathfrak H}}

\def\e{\mathrm{e}}
\def\i{\mathrm{i}}

\def\j{\operatorname{j{}}}

\def\GL{\operatorname{GL}}

\def\SL{\operatorname{SL}}

\def\PSL{\operatorname{PSL}}

\def\vol{\operatorname{vol}}

\def\bs{\backslash}

\def\wbar{{\overline w}}

\newcommand{\ind}[1]{\ensuremath{{\mathbbm{1}}{\big(#1\big)}}}

\newcommand{\beq}{\begin{equation}}
\newcommand{\eeq}{\end{equation}}

\newcommand{\myboldfont}{\mathbb}

\newcommand{\eps}{\varepsilon}
\renewcommand{\epsilon}{\varepsilon}
\renewcommand{\phi}{\varphi}
\renewcommand{\kappa}{\varkappa}
\renewcommand{\ge}{\geqslant}
\renewcommand{\le}{\leqslant}
\renewcommand{\leq}{\leqslant}
\renewcommand{\geq}{\geqslant}

\makeatletter   
\let\th\@undefined               
\makeatother

\DeclareMathOperator{\sh}{sh}
\DeclareMathOperator{\th}{th}
\DeclareMathOperator{\cth}{cth}
\DeclareMathOperator{\ch}{ch}

\DeclareMathOperator{\tg}{tg}

\DeclareMathOperator{\arctg}{arctg}

\newcommand{\SU}{\ensuremath{\mathrm{SU}}}

\newcommand{\PSU}{\ensuremath{\mathrm{PSU}}}

\newcommand{\quot}{\ensuremath{\backslash}}

\newcommand{\R}{\ensuremath{\myboldfont R}}

\newcommand{\Z}{\ensuremath{\myboldfont Z}}
\renewcommand{\H}{\ensuremath{\myboldfont H}}

\renewcommand{\d}{\ensuremath{\partial}}
\def\slim{s}

\title{Directions in hyperbolic lattices}
\author{Jens Marklof}
\author{Ilya Vinogradov}
\address{School of Mathematics, University of Bristol,
Bristol BS8 1TW, U.K.\newline
\hspace*{8pt}{\tt j.marklof@bristol.ac.uk}}
\address{School of Mathematics, University of Bristol,
Bristol BS8 1TW, U.K.\newline
\hspace*{8pt}{\tt ilya.vinogradov@bristol.ac.uk}}

\date{9 October 2014/31 July 2015.}
\thanks{The research leading to these results has received funding from the European Research Council under the European Union's Seventh Framework Programme (FP/2007-2013) / ERC Grant Agreement n. 291147.}

\begin{document}

\begin{abstract}
It is well known that the orbit of a lattice in hyperbolic $n$-space is uniformly distributed when projected radially onto the unit sphere. In the present work, we consider the fine-scale statistics of the projected lattice points, and express the limit distributions in terms of random hyperbolic lattices. This provides in particular a new perspective on recent results by Boca, Popa, and Zaharescu on 2-point correlations for the modular group, and by Kelmer and Kontorovich for general lattices in dimension $n=2$.
\end{abstract}

\maketitle

\section{Introduction}

Let $\HH^n$ denote hyperbolic $n$-space, and $G$ its group of orientation-preserving isometries. A discrete subgroup $\Gamma<G$ is called a \emph{lattice} if it has a finite volume fundamental domain. We denote by $\Gamma_w$ the stabilizer of $w\in\HH^n$ in $\Gamma$. Since $\Gamma$ acts properly discontinuously on $\HH^n$, $\Gamma_w$ is a finite group. Given a point $z\in\HH^n$ we define the \emph{direction} $\phi_z(w)$ of a point $w\in\HH^n\setminus \{z\}$ as the intersection of the semi-infinite geodesic ray starting at $z$ and passing through $w$ with the unit sphere $\SS_z^{n-1}=\{ w\in\HH^n:d(w,z)=1\}$ centered at $z$, where $d(w,z)$ is the hyperbolic distance between $z$ and $w$. 
 Due to the homogeneity of $\HH^n$, we may alternatively think of $\phi_z(w)$ as the unit tangent vector at $z$ which is tangent to the above geodesic ray, or as the ray's endpoint on the boundary of $\HH^n$. 

The goal of the present paper is to explain the statistical distribution of directions in the orbit $\wbar := \Gamma w$
within distance $t$ to a fixed observer at the point $z$,
\begin{equation}\label{direction1}
\scrP_{t}^z(\wbar):=\{ \phi_z(\gamma w) :  \gamma\in\Gamma/\Gamma_w , \; 0<d(\gamma w,z)\leq t  \} ,
\end{equation}
in the limit $t\to\infty$. Here $\scrP_{t}^z(\wbar)$ is defined as a \emph{multiset}; i.e., the directions are recorded with multiplicity so that 
\begin{equation}
\#\scrP_{t}^z(\wbar) = \#\{ \gamma\in\Gamma/\Gamma_w : 0<d(\gamma w,z)\leq t \} .
\end{equation}
Given any choice of origin $o\in\HH^n$ and any isometry $g\in G$ so that $gz=o$, we have 
\begin{equation}
\scrP_{t}^z(\wbar) = \scrP_{t}^o(g\wbar) .
\end{equation}
This allows us to consider instead the distribution of directions of the point set $g\wbar$ relative to the fixed origin $o$. We will in the following omit indicating the dependence on $o$ and write $\SS^{n-1}:=\SS_o^{n-1}$, $d(w):=d(w,o)$, etc.

It is natural to also consider the directions of lattice points in a spherical shell with outer radius $t$ and width $s\in(0,t)$, 
\begin{equation}\label{directions2}
\scrP_{t,s}(g\wbar) := \{ \phi(g\gamma w) :  \gamma\in\Gamma/\Gamma_w ,\; t-s < d(g\gamma w)\leq t \} ,
\end{equation}
again defined as a multiset.
To unify the notation for balls and shells we set $\scrP_{t,s} (g\wbar) :=\scrP_t (g\wbar)$ for $t\leq s\leq\infty$.

The volume of the unit sphere $\SS^{n-1}$ in $\HH^n$ is $\vol_{\SS^{n-1}}(\SS^{n-1}) =\Omega_n \sh^{n-1}1$,
where $\Omega_n:=\frac{2\pi^{n/2}}{\Gamma(\frac n2)}$ is the full solid angle (that is, the volume of the Euclidean unit sphere in $\RR^n$). As it is natural to measure directions in solid angles, we will in the following use the measure $\omega$ on $\SS^{n-1}$ defined by
$\omega(\scrA)=\vol_{\SS^{n-1}}(\scrA)/\sh^{n-1} 1$,
so that $\omega(\SS^{n-1})=\Omega_n$.

It is well known \cite{nicholls_lattice_1983, boca_distribution_2007} that the directions $\scrP_{t,s}(g\wbar)$ are uniformly distributed on $\SS^{n-1}$; i.e., for every $\scrA\subset\SS^{n-1}$ with boundary of measure zero and $s\in(0,\infty]$, we have
\begin{equation}\label{eq:UD}
\lim_{t\to\infty} \frac{\#(\scrP_{t,s} (g\wbar) \cap \scrA)}{\#\scrP_{t,s} (g\wbar)} = \frac{\omega(\scrA)}{\Omega_n} .
\end{equation}

For large $t$, the total number of points in the spherical shell has asymptotics 
\begin{equation}\label{ball-count}
\begin{split}
\#\scrP_{t,s} (g\wbar) & \sim \frac{\{ z\in\HH^n : t-\slim < d(z)\leq t\}}{\#\Gamma_w \vol_{\HH^n}(\Gamma\bs\HH^n)} \\
& \sim \Omega_n\, \vartheta\, \e^{(n-1)t} ,
\end{split}
\end{equation}
where
\begin{equation}\label{vartheta}
\vartheta:=\frac{1-\e^{-(n-1)\slim}}{(n-1) \#\Gamma_w \vol_{\HH^n}(\Gamma\bs\HH^n)} 
\end{equation}
and $\vol_{\HH^n}(\Gamma\bs\HH^n)$ denotes the volume of a fundamental domain of the $\Gamma$-action on $\HH^n$.

The challenge is to understand the fine-scale distribution of the point set $\scrP_{t,s}(g\wbar)$ for large $t$. One example of such a statistic is the 2-point correlation function, for which a limit formula was conjectured by Boca et al.\ \cite{boca_pair_2013}  in dimension $n=2$ and proved in the special case $\Gamma=\SL(2,\ZZ)$ and $z=\i$ or $z=\e^{\i\pi/3}$. The general proof of this conjecture was recently given by Kelmer and Kontorovich \cite{kelmer_pair_2013}. In the present paper we extend their limit theorems to general local statistics and to arbitrary dimension $n\geq 2$ by adapting the strategy developed in the Euclidean setting \cite{marklof_strombergsson_free_path_length_2010}. The key step is the reduction of convergence in distribution to equidistribution of large spheres in relevant moduli spaces. 

The recent work of Risager and S\"odergren \cite{risager_angles_2014} extends the effective convergence of the 2-point correlation function in \cite{kelmer_pair_2013} to arbitrary dimension $n\ge 2$; it also includes an explicit formula for the limit in dimension $n=3$. The analysis in  \cite{risager_angles_2014} is restricted to the 2-point correlations of distances between the projected points on $\SS^{n-1}$. The approach presented here yields 2-point (as well as higher order) correlations of both distances and relative orientation, as we permit test sets that are not rotationally invariant.

Given $\sigma>0$, denote by $\scrD_{t,s} (\sigma,\vecv)\subset\SS^{n-1}$ the open disc of volume 
\begin{equation}\label{eq:scaling}
\omega(\scrD_{t,s} (\sigma,\vecv)) = \Omega_n \, \frac{\sigma}{\#\scrP_{t,s} (g\wbar)}
\end{equation}
centered at a point $\vecv\in\SS^{n-1}$. When the denominator in the above equation vanishes, we set $\scrD_{t,s} (\sigma,\vecv)= \SS^{n-1}$. We are interested in the number of lattice directions in $\scrD_{t,s} (\sigma,\vecv)$,
\begin{equation}
\scrN_{t,s} (\sigma,\vecv;g\wbar) := \#(\scrP_{t,s} (g\wbar) \cap \scrD_{t,s} (\sigma,\vecv)),
\end{equation}
when $\vecv$ is distributed according to a fixed Borel probability measure $\lambda$ on $\SS^{n-1}$. 
We have chosen the volume of the disc in \eqref{eq:scaling} so that
\begin{equation}
\frac{1}{\Omega_n} \int_{\SS^{n-1}} \scrN_{t,s} (\sigma,\vecv;g\wbar) \, d\omega(\vecv) = \sigma.
\end{equation}
The asymptotic density \eqref{eq:UD} furthermore implies that, for any probability measure $\lambda$ with continuous density,  
\begin{equation}\label{eq:continuous_density}
\lim_{t\to\infty} 
\int_{\SS^{n-1}} \scrN_{t,s} (\sigma,\vecv;g\wbar) \, d\lambda(\vecv) = \sigma.
\end{equation}
(We will later show that this statement extends to $\lambda$ with bounded density, see Theorem \ref{th:moment-sph} below.)

The group $G$ acts on $G/\Gamma$ by left multiplication. In the following we denote by $\mu$ the unique $G$-invariant probability measure on $G/\Gamma$, which can be realized as the pushforward of the suitably normalized Haar measure on $G$ under the natural projection $G\to G/\Gamma$. We will also denote by $\mu$ the normalized Haar measure on $G$.

The following theorem is our principal result. 

\begin{thm}\label{th:distribution}
Let $\lambda$ be a Borel probability measure on $\SS^{n-1}$ absolutely continuous with respect to the Lebesgue measure.
Then, for every $r\in\ZZ_{\geq 0}$, $s\in(0,\infty]$ and $\sigma\in(0,\infty)$, 
\begin{equation}\label{th:distribution-eq1}
E_{\slim}(r,\sigma;\wbar):=\lim_{t\to\infty} \lambda( \{ \vecv\in\SS^{n-1} : \scrN_{t,s} (\sigma,\vecv;g\wbar) = r \}) 
\end{equation}
exists and is given by
\begin{equation}\label{lim:form}
E_{\slim}(r,\sigma;\wbar)=\mu(\{ h\in G/\Gamma  :  \#( h\wbar \cap \scrZ_0(\slim,\sigma) ) = r \}) ,
\end{equation}
where $\scrZ_0(\slim,\sigma) \subset\HH^n$ is a cuspidal cone defined in \eqref{zyl0-def} below.
The limit distribution $E_\slim(\cdot,\sigma;\wbar)$ is independent of $\lambda$ and $g$, continuous in $\slim\in(0,\infty]$ and $\sigma\in(0,\infty)$,
and satisfies
\begin{equation}\label{01law}
\lim_{\sigma\to 0} E_\slim(r,\sigma;\wbar) =
\begin{cases}
1 & (r=0) \\
0 & (r\geq 1).
\end{cases}
\end{equation} 
\end{thm}

If $\Gamma$ is co-compact, formula \eqref{lim:form} implies that for any $r\in\ZZ_{\geq 0}$, $s\in(0,\infty]$ there exists a constant $\sigma_0=\sigma_0(\Gamma,r,s)$ such that $E_{\slim}(r,\sigma;\wbar)=0$ for all $\sigma\in[\sigma_0,\infty)$. %This fact follows from the uniform density of large horodiscs in $\Gamma\backslash\HH^n$.

The proof of Theorem \ref{th:distribution} is given at the end of Section \ref{sec:Projection-sph}.
We will in fact extend this result in several ways:

\begin{itemize}[topsep=0.6ex, itemsep=0.7ex, parsep=0.6ex]
\item Instead of the number of points in a single disc, we will also consider the joint distribution in several test sets (not necessarily discs). These statistics capture all other local correlations, such as gap or nearest neighbor distributions.
\item We will prove convergence of mixed moments of all orders. This is in contrast to the Euclidean setting, where higher order moments  diverge \cite{EMV_directions_2013}. The second order mixed moment corresponds to the 2-point correlation function considered in dimension $n=2$ by Boca, Popa, and Zaharescu for the modular group \cite{boca_pair_2013} and by Kelmer and Kontorovich for general lattices \cite{kelmer_pair_2013}. 
\item If $\Gamma$ contains a parabolic subgroup $\Gamma_\infty$, it is natural to consider an observer positioned at the fixed point of $\Gamma_\infty$ on $\partial\HH^n$, the boundary of $\HH^n$. In this case, the directions correspond to the projections of the orbit $\Gamma w$ onto a closed horosphere in $\Gamma_\infty\bs\HH^n$. The uniform distribution of the projected orbit was proved by Good \cite{good_various_1983, risager_distribution_2010} (see also Rudnick and Risager  \cite{risager_rudnick_statistics_2009} for an interesting number-theoretic application). We will show that local statistics have the same limit distribution as in the non-cuspidal case.
\end{itemize}

This paper is organized as follows. Section \ref{sec:Hyperbolic} reviews some of the basic concepts of hyperbolic geometry that are used in our subsequent analysis. Section \ref{sec:Horo} comprises the equidistribution theorem for large horospheres, which is the key ingredient in the present study. The main results of this paper can be found in Sections \ref{sec:Projection-cusp} and \ref{sec:Projection-sph}, where we state and prove the convergence of the local statistics for directions observed by a cuspidal observer (Section \ref{sec:Projection-cusp}) and non-cuspidal observer (Section \ref{sec:Projection-sph}). The latter result requires equidistribution of large spheres, which is derived from the equidistribution for large horospheres in Section \ref{sec:Sphere}. The convergence of moments for cuspidal and non-cuspidal observers is discussed in Sections \ref{sec:Moments} and \ref{sec:Moments-sph}, respectively. We have arranged this paper to first give a full account in the cuspidal case (Sections \ref{sec:Horo}--\ref{sec:Moments}), which is technically simpler, and then explain the necessary adjustments for the non-cuspidal setting (Sections \ref{sec:Sphere}--\ref{sec:Moments-sph}). 

The Appendix shows how the 2-point correlation density is recovered from the second mixed moment, and reproduces the known formulas in dimension $n=2$.

\section{Hyperbolic geometry}\label{sec:Hyperbolic}

In dimension $n=2$, a convenient representation of the hyperbolic plane $\HH^2$ is given by the complex upper half-plane $\{ x+\i y : x\in\RR,\; y\in\RR_{>0}\}$. The advantage of the complex notation is that the action of the isometry group is given by M\"obius transformations
\begin{equation}
\HH^2\to\HH^2, \qquad z \mapsto \frac{az+b}{cz+d}, \qquad \begin{pmatrix} a & b \\ c & d \end{pmatrix} \in \SL(2,\RR).
\end{equation}
This model can be extended to higher dimension \cite{ahlfors_mobius_1985}, if one replaces complex numbers with Clifford numbers. We will here use the notation of Waterman \cite{waterman_mobius_1993} which is slightly different from Alfohrs' \cite{ahlfors_mobius_1985}. 

The \emph{Clifford algebra} $C_m$ is a real associative algebra generated by $\i_1, \i_2, \dots, \i_m$ subject to the conditions $\i_l^2=-1$  and $\i_j\i_l = - \i_l \i_j$ whenever $j\ne l$. Thus, every $a\in C_m$ can be expressed as \beq \label{eq:a_I_sum}a=\sum_I a_I I,\eeq where the sum ranges over all products $I=\i_{\nu_1}\cdots \i_{\nu_l}$ with $1\le \nu_1 < \dots <\nu_l \le m$ and $a_I\in \R$. The null product is also included and represents the real number $1$. The algebra $C_m$ forms a vector space of dimension $2^m$ over \R, and we take the norm $|a|^2 = \sum_I a_I^2$ on it. There are three useful involutions acting on $C_m$. The map $a\mapsto a'$ replaces every occurrence of $\i_l$ by $-\i_l$; it is an algebra automorphism. The map $a\mapsto a^*$ replaces each $I=\i_{\nu_1}\cdots \i_{\nu_l}$ in \eqref{eq:a_I_sum} by $\i_{\nu_l}\cdots \i_{\nu_1};$ it is an algebra anti-automorphism. The third involution is the composition of the first two, $a\mapsto a'^* =: \bar a$. 

The algebra contains special elements called \emph{Clifford vectors}, which are those of the form $\vecx = x_0 +x_1\i_1 +\dots +x_m \i_m$. We denote the corresponding vector space by $V_m$. We will identify $V_m$ with $\RR^{m+1}$ in the following via $\vecx\mapsto (x_0,\ldots,x_m)$. Clifford vectors satisfy $\vecx^*= \vecx$, $\bar \vecx = \vecx'$, and also $\vecx\bar \vecx= \bar \vecx \vecx=|\vecx|^2.$ In particular, non-zero vectors are invertible, since $\vecx^{-1}= \bar \vecx/|\vecx|^2$. Products of invertible vectors are also invertible and form a multiplicative group, $\Delta_m$, called the \emph{Clifford group}. 

We define the matrix groups
\begin{equation}
 \GL(2,C_m)  := \left\{\begin{pmatrix}\veca & \vecb\\ \vecc & \vecd\end{pmatrix} : \begin{aligned}\veca,\vecb,\vecc,\vecd & \in\Delta_m\cup \{0\}; \\ \veca \vecb^*, \vecc \vecd^*, \vecc^*\veca, \vecd^*\vecb &\in V_{m+1};\\ \veca \vecd^*-\vecb \vecc^* & \in \R\setminus\{0\}\end{aligned}\right\},
\end{equation}
\begin{equation}
\SL(2,C_m) := \left\{\begin{pmatrix}\veca & \vecb\\ \vecc & \vecd\end{pmatrix} \in \GL(2,C_m): \veca\vecd^*-\vecb \vecc^* = 1\right\},
\end{equation}
\begin{equation}
 \SU(2, C_m) := \left\{\begin{pmatrix} \veca & \vecb\\ -\vecb' & \veca'\end{pmatrix} \in \SL(2,C_m)\right\}.\label{eq:SUdef}
\end{equation}
We represent hyperbolic $n$-space as the upper half-space
\begin{equation}
\H^n=\{\vecx+ \j y  : \vecx\in V_{n-2},\; y\in \R_{>0}\}, \qquad \j:=\i_{n-1}.
\end{equation}
We will identify $\j$ as the origin $o$ in $\HH^n$ and, following the analogy with the two-dimensional setting, write  $\Re(z):=\vecx$ and $\Im(z):=y$ for the ``real'' and ``imaginary'' part of $z=\vecx+\j y $.
The Riemannian metric of $\H^n$ is defined by
\beq
ds^2 = \frac{\sum_{i=0}^{n-2} dx_i^2 +dy^2}{y^2}\label{eq:arclength} .
\eeq
The corresponding volume element is
\beq
d\!\vol_{\HH^n} = \frac{dx_0\cdots dx_{n-2} dy}{y^n}.\label{eq:volume}
\eeq
The action of $\SL(2,C_{n-2})$ on $\HH^n$ defined by the M\"obius transformation
\beq
\HH^n\to\HH^n,\qquad z\mapsto \begin{pmatrix}\veca & \vecb\\ \vecc & \vecd\end{pmatrix} z = (\veca z+\vecb)(\vecc z+\vecd)^{-1}
\eeq
(with multiplication from $C_{n-1}$) is isometric and orientation-preserving. Its kernel is $\{\pm 1\}$, so that 
\begin{equation}
G:=\PSL(2,C_{n-2})=\SL(2,C_{n-2})/ \{\pm 1\} 
\end{equation}
is isomorphic to the group of orientation-preserving isometries of $\HH^n$ \cite{waterman_mobius_1993}.
We also note that the M\"obius transformations preserve the boundary of hyperbolic space,
\begin{equation}
\partial\HH^n := V_{n-2} \cup \{ \infty \} . 
\end{equation}
The stabilizer of $\j$ under this action is 
\begin{equation}
K:=\PSU(2,C_{n-2})=\SU(2,C_{n-2})/\{\pm 1\} ,
\end{equation}
which is the maximal compact subgroup of $G$. Every $g\in G$ can be uniquely written as (Iwasawa decomposition)
\begin{equation}\label{Iwasawa}
g = n(\vecx) a(y) k
\end{equation}
with 
\begin{equation}
n(\vecx):=\begin{pmatrix} 1 & \vecx \\0&1\end{pmatrix} ,\qquad
a(y) := \begin{pmatrix}
y^{1/2} & 0\\ 0 & y^{-1/2} 
\end{pmatrix} ,\qquad
k\in K,
\end{equation}
and $\vecx\in V_{n-2}$, $y>0$. The Iwasawa decomposition yields a natural identification $\HH^n\cong G/K$. Thus $G$ can be represented as a frame bundle over $\HH^n$ with fiber $K$, which in dimension $n=2$ (only) can be identified with the unit tangent bundle of $\HH^2$.
The Haar measure on $G$ can thus be written as
\begin{equation}\label{iwa}
d\mu(g) = \kappa\, d\!\vol_{\HH^n}(z)\, dm(k)
\end{equation}
where $\kappa$ is a normalizing constant, and $m$ is the Haar probability measure on $K$. We have assumed above that $\mu$ is a probability measure on $G/\Gamma$. Since $G$ is unimodular we have $\mu(\Gamma\bs G)=1$ and hence $\kappa=\vol_{\HH^n}(\Gamma\bs\HH^n)^{-1}$.

\section{Equidistribution of large horospheres}\label{sec:Horo}

We now state the key equidistribution theorem from which all other results will follow. Let $\Gamma<G$ be a lattice, and let
\begin{equation}
\Phi^t=\begin{pmatrix}
\e^{t/2} & 0\\ 0 & \e^{-t/2} 
\end{pmatrix} .
\end{equation}
The one-parameter subgroup $\Phi^\RR:=\{ \Phi^t : t\in\RR\}$ acts by left multiplication on the coset space $G/\Gamma$. 
%This action is well known to be exponentially mixing \cite{ratner_rate_1987, moore_exponential_1987, pollicott_exponential_1992}. 
The horospherical subgroup $\{n(\vecx) : \vecx\in \RR^{n-1} \}$ parametrizes the unstable manifold of $\Phi^t$ for $t\to\infty$. The following theorem states that translates of horospheres become uniformly distributed in $G/\Gamma$.

\begin{theorem}\label{th:horospherical}
Let $\lambda$ be a Borel probability measure on $\RR^{n-1}$, 
absolutely continuous with respect to the Lebesgue measure.
Then, for any bounded continuous function $f:\RR^{n-1}\times G/\Gamma\to \R$
and any family of uniformly bounded continuous functions 
$f_t: \RR^{n-1}\times G/\Gamma\to \RR$ such that $f_t\to f$ as $t\to\infty$,
uniformly on compacta, and for every $g\in G$, we have
\begin{equation}\label{eq:horospherical}
 \lim_{t\to\infty} \int_{\R^{n-1}} f_t\left(\vecx, \Phi^t n(\vecx) g  \right) d\lambda(\vecx) = \int_{\RR^{n-1}\times G/\Gamma} f(\vecx,h)\,d\lambda(\vecx)\,d\mu(h).
\end{equation}
\end{theorem}

This theorem follows from the mixing property of the $\Phi^\RR$ action by an argument that goes back to Margulis' thesis \cite{margulis_some_aspects_2004}; see also the influential paper by Eskin and McMullen \cite{eskin_mixing_1993}. Precise rates of convergence are obtained by S\"odergren in the case of $K$-invariant functions \cite{sodergren_uniform_2012}. The test functions used in these papers usually do not depend on $t$ and $\vecx$. The extension to the formulation used here follows from a simple approximation argument, cf.\ the proof of Theorem 5.3 in \cite{marklof_strombergsson_free_path_length_2010}.
The extra $t$ and $\vecx$ dependence will be useful in proving the equidistribution of translates of spherical averages in $G/\Gamma$, see Section \ref{sec:Sphere}.

The following Corollary of Theorem \ref{th:horospherical} follows from the same argument as the proof of Theorem 5.6 in \cite{marklof_strombergsson_free_path_length_2010}. We recall the definition of limits of a family of sets 
$\{\scrE_t\}_{t\geq t_0}$ in $\R^{n-1}\times G/\Gamma$, where $t_0$ is a fixed real constant:
\begin{align}
	\liminf \scrE_t &:= \bigcup_{t\geq t_0} \bigcap_{s\geq t} \scrE_s , &
	\limsup \scrE_t &:= \bigcap_{t\geq t_0} \bigcup_{s\geq t} \scrE_s .\\
\intertext{We furthermore define}
\label{LIMINFSUPTOP}
	\lim(\inf \scrE_t)^\circ &:= \bigcup_{t\geq t_0} \bigg(\bigcap_{s\geq t} \scrE_s\bigg)^\circ , &
	\lim\overline{\sup \scrE_t} &:= \bigcap_{t\geq t_0} \overline{\bigcup_{s\geq t} \scrE_s} .
\end{align}
Note that $\lim(\inf \scrE_t)^\circ$ is open and $\lim\overline{\sup \scrE_t}$ is closed. $\chi_{\scrE}$ denotes the indicator function of the set $\scrE$; i.e., $\chi_{\scrE}(x)=1$ if $x\in\scrE$ and $\chi_{\scrE}(x)=0$ otherwise.

\begin{cor}\label{charThm}
Let $\lambda$ be a Borel probability measure on $\RR^{n-1}$, 
absolutely continuous with respect to the Lebesgue measure. Then, for any family of subsets $\scrE_t\subset\RR^{n-1}\times G/\Gamma$ and any $g\in G/\Gamma$, we have
\begin{equation}\label{inf}
	\liminf_{t\to\infty} \int_{\RR^{n-1}} \chi_{\scrE_t}(\vecx,\Phi^t n(\vecx) g) \,d\lambda(\vecx) \geq \int_{\lim(\inf \scrE_t)^\circ} d\lambda\,d\mu ,
\end{equation}
and
\begin{equation}\label{sup}
	\limsup_{t\to\infty} \int_{\RR^{n-1}} \chi_{\scrE_t}(\vecx,\Phi^t n(\vecx) g)\,  d\lambda(\vecx) \leq \int_{\lim\overline{\sup \scrE_t}} d\lambda\,d\mu.
\end{equation}
If furthermore $\lambda\times\mu$ gives zero measure to the set $\lim\overline{\sup \scrE_t}\setminus\lim(\inf \scrE_t)^\circ$, then
\begin{equation}\label{lim}
	\lim_{t\to\infty} \int_{\RR^{n-1}}\chi_{\scrE_t}(\vecx,\Phi^t n(\vecx) g)\,  d\lambda(\vecx) = \int_{\limsup \scrE_t} d\lambda\,d\mu .
\end{equation}
\end{cor}

\section{Projection statistics for cuspidal observer}\label{sec:Projection-cusp}

We assume in this section that the lattice $\Gamma<G$ contains a parabolic subgroup and hence also a maximal parabolic subgroup $\Gamma_\infty$. By conjugating $\Gamma$ by a suitable element of $G$, we may assume without loss of generality that 
\begin{equation}
\Gamma_\infty = \{ n(\vecm) : \vecm\in\scrL \},
\end{equation}
where $\scrL$ is a lattice in $\RR^{n-1}$ with covolume one. Geometrically, this means that the hyperbolic orbifold $\Gamma\backslash\HH^n$ has a cusp at $\infty$, whose cross-section  we identify with the torus $\TT^{n-1}=\RR^{n-1}/\scrL$ of volume one. $\Gamma_\infty$ is the parabolic stabilizer of the cusp at $\infty$; in dimension $n>2$ there may also be elliptic elements in $\Gamma$ that leave $\infty$ invariant, but we will not include these in $\Gamma_\infty$.

We now position our observer at the cusp at $\infty$ and consider---instead of radial projections \eqref{directions2}---the following vertical projections of the orbit $\wbar=\Gamma w$ onto the torus $\TT^{n-1}$:
\begin{equation}
\scrP_{t,s}^\infty(\wbar):=\{ \Re(\gamma w) :  \gamma\in \Gamma_\infty\backslash\Gamma/\Gamma_w , \; \e^{-t}\leq \Im(\gamma w)<\e^{s-t}  \} ,
\end{equation}
again considered as a multiset. It follows from the work of Good \cite{good_local_1983} that, for $s$ fixed,  $\scrP_{t,s}^\infty(\wbar)$ is uniformly distributed on $\TT^{n-1}$. That is, for every $\scrA\subset\TT^{n-1}$ with boundary of measure zero, we have
\begin{equation}\label{eq:UD2}
\lim_{t\to\infty} \frac{\#(\scrP_{t,s}^\infty(\wbar) \cap \scrA)}{\#\scrP_{t,s}^\infty(\wbar)} = \vol_{\TT^{n-1}}(\scrA).
\end{equation}
We recall also that for $t\to\infty$
\begin{equation}\label{Nasy}
\#\scrP_{t,s}^\infty(\wbar) \sim \vartheta \, \e^{(n-1)t}, 
\end{equation}
with $\vartheta$ as in \eqref{vartheta}.

To measure the fine-scale statistics of $\scrP_{t,s}^\infty(\wbar)$, consider the following rescaled test sets in $\TT^{n-1}$,
\begin{equation}\label{Ats}
\scrB_{t,s}(\scrA,\vecx) = N^{-1/(n-1)} \scrA -\vecx+ \scrL \subset \TT^{n-1}, \qquad N=\#\scrP_{t,s}^\infty(\wbar) ,
\end{equation}
where $\scrA\subset\RR^{n-1}$ is a fixed bounded set with boundary of Lebesgue measure zero. The shift $\vecx\in\TT^{n-1}$ is assumed to be random according to some probability measure $\lambda$. The random variable we use to detect correlations in the directions $\scrP_{t,s}^\infty(\wbar)$ is the number of points in $\scrB_{t,s}(\vecx)$,
\begin{equation}\label{Nts}
\scrN_{t,s}^\infty (\scrA,\vecx;\wbar) := \#(\scrP_{t,s}^\infty  (\wbar) \cap \scrB_{t,s} (\scrA,\vecx)).
\end{equation}
If $\lambda=\vol_{\TT^{n-1}}$ is the normalized Lebesgue measure, we have
\begin{equation}\label{expect220}
\int_{\TT^{n-1}} \scrN_{t,s}^\infty (\scrA,\vecx;\wbar)\, d\!\vol_{\TT^{n-1}}(\vecx) = \vol_{\RR^{n-1}} \scrA .
\end{equation}
Uniform distribution of $\scrP^\infty_{t,s}(\wbar)$ \eqref{eq:UD2} implies that for any $\lambda$ with continuous density, the expectation value of $\scrN_{t,s}^\infty (\scrA,\vecx;\wbar)$ converges:
\begin{equation}\label{expect22}
\lim_{t\to\infty} \int_{\TT^{n-1}} \scrN_{t,s}^\infty (\scrA,\vecx;\wbar)\, d\lambda(\vecx) = \vol_{\RR^{n-1}} \scrA .
\end{equation}

The following theorem considers convergence in distribution for several test sets $\scrA_1,\ldots,\scrA_m$:

\begin{thm}\label{th:distributionA}
Let $\lambda$ be a Borel probability measure on $\TT^{n-1}$ absolutely continuous with respect to the Lebesgue measure. Then, for every $r=(r_1,\ldots,r_m)\in\ZZ_{\geq 0}^m$, $\slim\in(0,\infty]$ and $\scrA=\scrA_1\times\cdots\times\scrA_m$ with $\scrA_j\subset\RR^{n-1}$ bounded with boundary of Lebesgue measure zero,
\begin{equation}
E_{\slim}(r,\scrA;\wbar):=\lim_{t\to\infty} \lambda( \{ \vecx\in\TT^{n-1} : \scrN_{t,s}^\infty (\scrA_j,\vecx;\wbar) = r_j \;\forall j \}) 
\end{equation}
exists and is given by
\begin{equation}\label{limit-dist}
E_{\slim}(r,\scrA;\wbar)=\mu(\{ g\in G/\Gamma  :  \#( g\wbar \cap \scrZ(\slim,\scrA_j) ) = r_j \;\forall j\}) ,
\end{equation}
where
\begin{equation}\label{cyldef0}
\scrZ(\slim,\scrA_j) := \{ z\in\HH^n : \Re z \in \vartheta^{-1/(n-1)}\scrA_j,\; 1\leq \Im z <\e^\slim \}.
\end{equation}
The limit distribution $E_{\slim}(r,\scrA;\wbar)$ is independent of $\lambda$, and continuous in $\slim$ and $\scrA$.
\end{thm}

By ``{\em continuous in $\scrA$}'' we mean here more specifically that there is a constant $C$ such that
\begin{equation}\label{conti}
 \big| E_{\slim}(r,\scrA;\wbar) - E_{\slim}(r,\scrB;\wbar) \big| \leq C  \vol_{\RR^{m(n-1)}}(\scrB\setminus\scrA)
\end{equation}
for all product sets $\scrA\subset \scrB\subset\RR^{m(n-1)}$ as in Theorem \ref{th:distributionA}.

The following lemma implies the continuity asserted in Theorem \ref{th:distributionA}.
For $\scrA\subset \H^n$ and $r\in\Z_{\ge 0}$, write
\begin{align}\label{eq:rpoints}
[\scrA]_{\le r} & := \{g\in G/\Gamma  : \#(\scrA\cap g\wbar)\le r\} \\
[\scrA]_{= r} & := \{g\in G/\Gamma  : \#(\scrA\cap g\wbar) = r\} \\
[\scrA]_{\ge r} & := \{g\in G/\Gamma  : \#(\scrA\cap g\wbar)\ge r \} .
\end{align}

\begin{lemma}\label{lem:small_sets}
For any measurable $\scrA\subset \scrB\subset \H^n$ with finite volume, we have
\beq \label{eq1:small_sets}
\mu ([\scrA]_{\ge 1}) \le \frac{\vol_{\H^n} \scrA}{\#\Gamma_w \vol_{\HH^n}(\Gamma\bs\HH^n)} ,
\eeq
\beq \label{eq2:small_sets}
|\mu ([\scrA]_{=r}) - \mu ([\scrB]_{=r})|  \le \frac{\vol_{\H^n} (\scrB\setminus \scrA) }{\#\Gamma_w \vol_{\HH^n}(\Gamma\bs\HH^n)} 
\eeq
and
\beq \label{eq3:small_sets}
0\leq \mu ([\scrA]_{\leq r}) - \mu ([\scrB]_{\leq r})  \le \frac{\vol_{\H^n} (\scrB\setminus \scrA) }{\#\Gamma_w \vol_{\HH^n}(\Gamma\bs\HH^n)} .
\eeq
\end{lemma}

\begin{proof}
By Chebyshev's inequality, 
\beq
\mu \{g\in G/\Gamma  : \#(\scrA \cap g\wbar) \ge 1\}
\le \int_{G/\Gamma} \#(\scrA\cap g\wbar) d\mu(g).
\eeq
Combining the integral over $G/\Gamma$ with the sum over $\Gamma$ gives, by the standard unfolding technique,
\begin{equation}\label{thesame}
\begin{split}
\int_{G/\Gamma} \#(\scrA\cap g\wbar) d\mu(g) 
& = \int_{G/\Gamma} \sum_{\gamma\in\Gamma/\Gamma_w} \chi_\scrA(g\gamma w)d\mu(g)\\
&=\frac{1}{\#\Gamma_w}  \int_{G}  \chi_\scrA(gw) d\mu(g) \\
&=\frac{1}{\#\Gamma_w}  \int_{G}  \chi_\scrA(g\j) d\mu(g) \\
&= \frac{\vol_{\H^n} \scrA}{\#\Gamma_w \vol_{\HH^n}(\Gamma\bs\HH^n)} 
\end{split}
\end{equation}
in view of \eqref{iwa}. This proves \eqref{eq1:small_sets}. Relations \eqref{eq2:small_sets} and \eqref{eq3:small_sets} follow from the inequalities
\beq \label{eq2b:small_sets}
|\mu ([\scrA]_{=r}) - \mu ([\scrB]_{=r})|  \le \mu ([\scrB\setminus \scrA]_{\ge 1}) 
\eeq
and
\beq \label{eq3b:small_sets}
0\leq \mu ([\scrA]_{\leq r}) - \mu ([\scrB]_{\leq r})  \le \mu ([\scrB\setminus \scrA]_{\ge 1}) .
\eeq
\end{proof}

\begin{lemma}\label{sat-lem}
Under the hypotheses of Theorem \ref{th:distributionA}, given $\epsilon>0$ there exist $t_0$ and bounded subsets $\scrA_j^-,\scrA_j^+\subset\RR^{n-1}$ with boundary of measure zero such that 
\begin{equation}
\scrA_j^-\subset\scrA_j\subset\scrA_j^+, \qquad \vol_{\RR^{n-1}}(\scrA_j^+\setminus\scrA_j^-)<\epsilon
\end{equation}
and, for all $t\geq t_0$,
\begin{equation}
\#( \Phi^t n(\vecx) \wbar \cap \scrZ(\slim,\scrA_j^-) ) \leq \scrN_{t,s}^\infty (\scrA_j,\vecx;\wbar)  
\leq  \#( \Phi^t n(\vecx) \wbar \cap \scrZ(\slim,\scrA_j^+) ) .
\end{equation}
\end{lemma}

\begin{proof}
In view of \eqref{Ats} we have
\begin{equation}\label{eq:exact_count}
\scrN_{t,s}^\infty (\scrA_j,\vecx;\wbar)  = \#( \Phi^t n(\vecx) \wbar \cap \scrZ(s, \e^{t} \vartheta^{1/(n-1)}N^{-1/(n-1)} \scrA_j) ) .
\end{equation}
The asymptotics \eqref{Nasy} shows that $ \e^{t} \vartheta^{1/(n-1)}N^{-1/(n-1)}\to 1$ and hence the lemma.
\end{proof}

We extend the definition of the cuspidal cone in \eqref{cyldef0} to
\begin{equation}\label{cyldef007}
\scrZ(a,b,\scrA_j) := \{ z\in\HH^n : \Re z \in \vartheta^{-1/(n-1)} \scrA_j,\; \e^a\leq \Im z < \e^b \},
\end{equation}
where $-\infty<a<b\leq\infty$.

\begin{lemma}
 \label{lem:truncation}
 Under the hypotheses of Theorem \ref{th:distributionA}, we have for all $s\ge 0$ 
 \begin{multline}\label{eq:truncation}
 \limsup_{t\to\infty} 
 \big\lvert\lambda(\{\vecx\in \TT^{n-1}  : \#( \Phi^t n(\vecx) \wbar \cap \scrZ(\infty,\scrA_j) ) \le r_j\,\forall j \}) \\ -
 \lambda(\{\vecx\in \TT^{n-1} :  \#( \Phi^t n(\vecx) \wbar \cap \scrZ(\slim,\scrA_j) ) \le r_j\,\forall j \})\big\rvert
 \le  \e^{-(n-1)s/2} (\vol_{\R^{n-1}}\widetilde\scrA)^{1/2} 
\end{multline}
where $\widetilde\scrA=\cup_j\scrA_j$.
\end{lemma}

\begin{proof}
Using equation \eqref{Nts} we bound from above the left hand side of \eqref{eq:truncation}, without the $\limsup$, by
\begin{multline}
\lambda(\{ \vecx\in\TT^{n-1} : \#( \Phi^t n(\vecx) \wbar \cap \scrZ(\slim,\infty,\scrA_j) ) \ge 1\, \text{for some $j$}\}) \\
=
\lambda(\{ \vecx\in\TT^{n-1} : \#( \Phi^t n(\vecx) \wbar \cap \scrZ(\slim,\infty,\widetilde\scrA) ) \ge 1\}) .
\label{eq:together}
\end{multline}
Note that
\begin{equation}
\begin{split}
\#( \Phi^t n(\vecx) \wbar \cap \scrZ(\slim,\infty,\widetilde\scrA) )
& = \#( \Phi^{t-\slim} n(\vecx) \wbar \cap \scrZ(0,\infty,\e^{-\slim}\widetilde\scrA) ) \\
& = \scrN_{t-s,\infty}^\infty (\eta_{t,s} \e^{-s} \widetilde\scrA,\vecx;\wbar),
\end{split}
\end{equation}
where $\eta_{t,s} \to 1$ as $t\to\infty$.

For any $R\geq 1$, Chebyshev's inequality implies the bound
\begin{align}
\label{eq:choose_j}\eqref{eq:together} \le \int_{\TT^{n-1}} \scrN_{t-s,\infty}^\infty (\eta_{t,s} \e^{-s} \widetilde\scrA,\vecx;\wbar)\, d\lambda_R(\vecx) +\frac1R,
\end{align}
where $\lambda_R$ is the Borel measure with density $\lambda_R'(\vecx):=\min(\lambda'(\vecx),R)$. 
In view of eq.~\eqref{expect220},
\begin{equation}
\begin{split}
& \int_{\TT^{n-1}} \scrN_{t-s,\infty}^\infty (\eta_{t,s} \e^{-s} \widetilde\scrA,\vecx;\wbar)\, d\lambda_R(\vecx) \\
& \leq R \int_{\TT^{n-1}} \scrN_{t-s,\infty}^\infty (\eta_{t,s} \e^{-s} \widetilde\scrA,\vecx;\wbar)\, d\!\vol_{\TT^{n-1}}(\vecx) \\
& = R\, \eta_{t,s}^{n-1} \e^{-(n-1)s} \vol_{\RR^{n-1}}\widetilde\scrA \\
& \to R\, \e^{-(n-1)s} \vol_{\RR^{n-1}}\widetilde\scrA 
\end{split}
\end{equation}
as $t\to\infty$.
The lemma is obtained by choosing $R=(\e^{-(n-1)s}\vol_{\RR^{n-1}}\widetilde\scrA )^{-1/2}$.
\end{proof}

\begin{proof}[Proof of Theorem \ref{th:distributionA}]
It is sufficient to show that, for every $r=(r_1,\ldots,r_m)\in\ZZ_{\geq 0}^m$ and $\scrA=\scrA_1\times\cdots\times\scrA_m$ with $\scrA_j\subset\RR^{n-1}$ bounded with boundary of Lebesgue measure zero,
\begin{align}\label{eq:convergence_in_distribution}
\lambda( \{ \vecx\in\TT^{n-1} : \scrN_{t,s}^\infty (\scrA_j,\vecx;\wbar) \leq r_j \;\forall j \}) 
\to \mu(\{ g\in G/\Gamma  :  \#( g\wbar \cap \scrZ(\slim,\scrA_j) ) \leq r_j \;\forall j\}) 
\end{align}
as $t \to\infty$. 

The left hand side equals
\begin{equation}
\int_{\TT^{n-1}} \chi_{\scrE_t}(\vecx, \Phi^t n(\vecx)) \,d\lambda(\vecx)
\end{equation}
with 
\begin{equation}
\scrE_t = \TT^{n-1}\times \{ g\in G/\Gamma  :   \#( g\wbar \cap \scrZ(s, \e^{t} \vartheta^{1/(n-1)}N^{-1/(n-1)} \scrA_j)) \leq r_j \;\forall j\} .
\end{equation}
The plan is now to apply Corollary \ref{charThm} (where we identify $\lambda$ with a probability measure on $\RR^{n-1}$ supported on a fundamental domain of $\scrL$). 
Given any $\eps>0$, define (with the notation as in Lemma \ref{sat-lem})
\begin{equation}
\scrE_s^\pm = \TT^{n-1}\times\{ g\in G/\Gamma  :   \#( g\wbar \cap \scrZ(\slim,\scrA_j^\pm) ) \leq r_j \;\forall j\} 
\end{equation}
and note that $\scrE_s^+\subset \scrE_t \subset \scrE_s^-$ for all $t\geq t_0$. We assume first $s<\infty$. Then Corollary \ref{charThm} yields
\begin{equation}
\limsup_{t\to\infty}\int_{\TT^{n-1}} \chi_{\scrE_t}(\vecx, \Phi^t n(\vecx)) \,d\lambda(\vecx)
\leq \mu(\overline{\scrE_s^-})
\end{equation}
and
\begin{equation}
\liminf_{t\to\infty}\int_{\TT^{n-1}} \chi_{\scrE_t}(\vecx, \Phi^t n(\vecx)) \,d\lambda(\vecx)
\geq \mu((\scrE_s^+)^\circ) .
\end{equation}
Lemmas \ref{lem:small_sets} and \ref{sat-lem} together with the fact that $\scrZ(s,\scrA_j^\pm)$ is bounded when $s<\infty$ imply that $\lim_{\epsilon\to 0} \mu(\overline{\scrE_s^-}\setminus(\scrE_s^+)^\circ)=0$. This proves \eqref{eq:convergence_in_distribution} for $s<\infty$.

Let us assume now that $s=\infty$. Given $\eps>0$ and $\scrA_j^\pm$ as above, by Lemma \ref{lem:truncation} there exists $s_\eps<\infty$ so that 
\begin{equation}
\limsup_{t\to\infty}\int_{\TT^{n-1}} \chi_{\scrE_t}(\vecx, \Phi^t n(\vecx)) \,d\lambda(\vecx)
\leq \mu(\overline{\scrE_{s_\eps}^-}) +\eps
\end{equation}
and
\begin{equation}
\liminf_{t\to\infty}\int_{\TT^{n-1}} \chi_{\scrE_t}(\vecx, \Phi^t n(\vecx)) \,d\lambda(\vecx)
\geq \mu((\scrE_{s_\eps}^+)^\circ) -\eps .
\end{equation}
Again, by Lemmas \ref{lem:small_sets} and \ref{sat-lem}, we have
\begin{equation}
\lim_{\eps\to 0} \mu(\overline{\scrE_{s_\eps}^-}) = \lim_{\eps\to 0} \mu((\scrE_{s_\eps}^+)^\circ) =
 \mu(\{ g\in G/\Gamma  :  \#( g\wbar \cap \scrZ(\infty,\scrA_j) ) \leq r_j \;\forall j\}) 
\end{equation}
which completes the proof of  \eqref{eq:convergence_in_distribution} for $s=\infty$.
\end{proof}

\begin{remark}
Note that for $\scrA\subset\RR^{n-1}$
\begin{equation}\label{eq:volume_cylinder}
\vol_{\HH^n} \scrZ(\slim,\scrA) = \frac{1-\e^{-(n-1)\slim}}{(n-1) \vartheta}\, \vol_{\RR^{n-1}} \scrA 
= \#\Gamma_w \vol_{\HH^n}(\Gamma\bs\HH^n) \, \vol_{\RR^{n-1}} \scrA .
\end{equation}
By the same calculation as in \eqref{thesame} we have for the expectation value of the limit distribution,
\begin{equation}\label{thesameagain}
\begin{split}
\sum_{r=0}^\infty r E_{\slim}(r,\scrA;\wbar) & = \int_{G/\Gamma} \#( g\wbar \cap \scrZ(\slim,\scrA) )\, d\mu(g)  \\
& = \frac{\vol_{\HH^n} \scrZ(\slim,\scrA)}{\#\Gamma_w \vol_{\HH^n}(\Gamma\bs\HH^n)} = \vol_{\RR^{n-1}} \scrA ,
\end{split}
\end{equation}
which is, of course, consistent with \eqref{expect22}.
\end{remark}

\begin{remark}
Theorem \ref{th:distributionA} has the following interpretation in the language of random point processes. For each $t\geq 0$, the point set
\begin{equation}
\Xi_t := N^{1/(n-1)} (\scrP_{t,s}^\infty(\wbar) +\vecx+\scrL) 
\end{equation}
defines, with the random variable $\vecx\in\TT^{n-1}$ distributed according to $\lambda$, a random point process on $\RR^{n-1}$. Theorem \ref{th:distributionA} says that this point process converges in finite-dimensional distribution to a random point process on $\RR^{n-1}$ defined by
\begin{equation}
\Xi := \{ \vartheta^{1/(n-1)}\Re(g\gamma w) :  \gamma\in\Gamma/\Gamma_w,\; 1\leq \Im(g\gamma w) <\e^\slim \},
\end{equation}
where $g\in G/\Gamma$ is a random variable distributed according to $\mu$. In view of \eqref{thesameagain}, the limit process has intensity one and, due to the $G$-invariance of $\mu$, is invariant under all translations and rotations of $\RR^{n-1}$. 
\end{remark}

\section{Convergence of moments for cuspidal observer}\label{sec:Moments}

Given bounded test sets $\scrA_1,\ldots,\scrA_m\subset\RR^{n-1}$ as above, we define the moment generating function
\begin{equation}
\GG_{t,s}^\infty(\tau_1,\ldots, \tau_m;\scrA) := \int_{\TT^{n-1}} \exp\bigg(\sum_{j=1}^m \tau_j \scrN_{t,s}^\infty (\scrA_j,\vecx;\wbar) \bigg) d\lambda(\vecx) ,
\end{equation}
which is analytic in all variables, and the moment generating function of the limit distribution,
\begin{equation}\label{limit-G}
\GG_{\slim}(\tau_1,\ldots, \tau_m;\scrA) := \sum_{r_1,\ldots,r_m=0}^\infty \exp\bigg(\sum_{j=1}^m \tau_j r_j\bigg) E_{\slim}(r,\scrA;\wbar) .
\end{equation}

We denote the positive real part of a complex number $\tau$ by $\Re_+\tau:=\max(\Re \tau, 0)$.

\begin{thm}\label{th:moment-cusp}
Let $\lambda$ be a probability measure on $\TT^{n-1}$ with bounded density, and $\scrA=\scrA_1\times\cdots\times\scrA_m$ with $\scrA_j\subset\RR^{n-1}$ bounded with boundary of Lebesgue measure zero. Then there is a constant $c_0>0$ such that for $\Re_+\tau_1+\ldots + \Re_+\tau_m< c_0$, $\slim\in(0,\infty]$, 
\begin{enumerate}
\item[(i)] $\GG_{\slim}(\tau_1,\ldots, \tau_m;\scrA)$ is analytic,
\item[(ii)] $\lim_{t\to\infty} \GG_{t,s}^\infty(\tau_1,\ldots, \tau_m;\scrA)  = \GG_{\slim}(\tau_1,\ldots, \tau_m;\scrA)$.
\end{enumerate}
\end{thm}

This theorem implies, by a standard argument, convergence of mixed moments of the form 
\begin{equation}
\MM_{t,s}^\infty(\beta_1,\ldots, \beta_m;\scrA) := \int_{\TT^{n-1}} \prod_{j=1}^m \left(\scrN_{t,s}^\infty (\scrA_j,\vecx;\wbar)\right) ^{\beta_j} d\lambda(\vecx) 
\end{equation}
for all $\beta_j\in\RR_{\geq 0}$. The corresponding limit moment is 
\begin{equation}\label{def:moment}
\MM_{\slim}(\beta_1,\ldots, \beta_m;\scrA) := \sum_{r_1,\ldots,r_m=0}^\infty r_1^{\beta_1}\cdots r_m^{\beta_m}  E_{\slim}(r,\scrA;\wbar) .
\end{equation}

\begin{cor}\label{cor:moment-cusp}
Let $\lambda$ be a probability measure on $\TT^{n-1}$ with bounded density, and $\scrA=\scrA_1\times\cdots\times\scrA_m$ with $\scrA_j\subset\RR^{n-1}$ bounded with boundary of Lebesgue measure zero. Then, for all $\beta_1,\ldots,\beta_m\in\RR_{\geq 0}$, $\slim\in(0,\infty]$, 
\begin{enumerate}
\item[(i)] $\MM_{\slim}(\beta_1,\ldots, \beta_m;\scrA)<\infty$,
\item[(ii)] $\lim_{t\to\infty} \MM_{t,s}^\infty(\beta_1,\ldots, \beta_m;\scrA)   = \MM_{\slim}(\beta_1,\ldots, \beta_m;\scrA)$.
\end{enumerate}
\end{cor}

Moreover it follows from \eqref{limit-dist} that all moments with $\beta_j\in\NN$ can be computed from explicit formulas; for $\beta_1=1,\ldots,\beta_m=1$, the formula reads
\begin{align}\label{eq:mixmo}
\MM_{\slim}(1,\ldots,1;\scrA) & = \int_{G/\Gamma} \sum_{\gamma_1,\ldots,\gamma_m\in\Gamma/\Gamma_w}
\prod_{j=1}^m \ind{g\gamma_j w\in\scrZ(\slim,\scrA_j)} \,d\mu(g)\\
& = \frac{1}{\#\Gamma_w} \sum_{\gamma_1,\ldots,\gamma_{m-1}\in\Gamma/\Gamma_w}
\int_{G}  \prod_{j=1}^{m-1} \ind{g\gamma_j w\in\scrZ(\slim,\scrA_j)} \ind{g w\in\scrZ(\slim,\scrA_m)} \,d\mu(g), 
\label{eq:mixmo2}
\end{align}
and, as previously noted \eqref{thesameagain},
\begin{equation}\label{five-six}
\MM_{\slim}(1;\scrA) = \frac{1}{\#\Gamma_w} \int_{G} \ind{g w\in\scrZ(\slim,\scrA)} \,d\mu(g) 
= \frac{\vol_{\HH^n} \scrZ(\slim,\scrA)}{\#\Gamma_w \vol_{\HH^n}(\Gamma\bs\HH^n)} = \vol_{\RR^{n-1}} \scrA.
\end{equation}
The Boolean function $\ind{\mathrm{B}}$ is defined by
\begin{equation}
\ind{\mathrm{B}}=\begin{cases}
1 & \text{if $\mathrm{B}=\mathrm{TRUE}$}\\
0 & \text{if $\mathrm{B}=\mathrm{FALSE}$.}
\end{cases}
\end{equation}

\begin{remark}\label{rem:2pt}
The convergence of the 2-point correlation function follows from Corollary \ref{cor:moment-cusp} for the second mixed moment ($m=2$) by choosing as test set $\scrA=\scrA'\times\scrB_\epsilon$, where $\scrB_\epsilon$ is a ball of small radius $\epsilon\to 0$. This is a fairly standard argument, cf.\ for instance \cite[Appendix 1]{EMV_directions_2013}. We will show in the Appendix that this recovers the known formulas in dimension $n=2$. 
\end{remark}

The proof of Theorem \ref{th:moment-cusp} will exploit the following three lemmas.
Let 
\begin{equation}
\delta(g\wbar):= \min_{\substack{\gamma_1,\gamma_2\in\Gamma\\ \gamma_1\notin\gamma_2\Gamma_w}} d(g\gamma_1 w,g \gamma_2 w).
\end{equation}
Since $G$ acts by isometries, we have in fact
\begin{equation}
\delta(g\wbar)=\min_{\gamma\in\Gamma\setminus\Gamma_w} d(w,\gamma w)=\delta(\wbar).
\end{equation}
Since $G$ acts properly discontinuously, we have $\delta(\wbar)>0$. 

If $-\infty<a<\slim<\infty$ and $\scrA\subset\RR^{n-1}$ is bounded, then $\scrZ(a,\slim,\scrA)$ is bounded. Hence the number of points in $g\wbar\cap \scrZ(a,\slim,\scrA)$ has an upper bound which is uniform in $g$. This in turn implies that all moments converge. The interesting case is thus $\slim=\infty$.

\begin{lem}\label{lem-1}
Fix $a\in\RR$ and a bounded subset $\scrA\subset\RR^{n-1}$. There exist positive constants $\zeta,\eta$ so that for all $g\in G$, $r\in\NN$,
\begin{equation}\label{ii1}
\big[ \#( g\wbar \cap \scrZ(a,\infty,\scrA)) \geq r \big] \quad \Rightarrow \quad \big[\#( g\wbar \cap \scrZ(\zeta r -\eta,\infty,\scrA)) \geq 1\big].
\end{equation}
\end{lem}

\begin{proof}
Let $\slim_0=2\delta(\wbar)^{-1} \diam\scrA$. Note that $d(\vecx_1+ \j \e^{\slim_0},\vecx_2+ \j \e^{\slim_0}) \leq \frac12 \delta(\wbar)$ for all $\vecx_1,\vecx_2\in\scrA$.
By the triangle inequality, if $z_1,z_2\in g\wbar \cap \scrZ(\slim_0,\infty,\scrA)$, then $\lvert\Im z_1-\Im z_2\rvert \geq \e^{\frac12 \delta(\wbar)}$. This proves for all $g\in G$, $r\in\NN$
\begin{equation}
\big[ \#( g\wbar \cap \scrZ(\slim_0,\infty,\scrA)) \geq r \big] \quad \Rightarrow \quad \big[\#( g\wbar \cap \scrZ(\slim_0 + \tfrac12 \delta(\wbar) (r-1),\infty,\scrA)) \geq 1\big].
\end{equation}
If $r_0:=\sup_{g\in G}\#( g\wbar \cap \scrZ(a,\slim_0,\scrA))$, we obtain for all $g\in G$, $r\in\NN$,
\begin{equation}
\big[ \#( g\wbar \cap \scrZ(a,\infty,\scrA)) \geq r + r_0 \big] \quad \Rightarrow \quad \big[\#( g\wbar \cap \scrZ(\slim_0 + \tfrac12 \delta(\wbar) (r-1),\infty,\scrA)) \geq 1\big],
\end{equation}
which implies \eqref{ii1} with $\zeta= \tfrac12 \delta(\wbar)$ and $\eta=\max(\zeta (r_0+1) - \slim_0, \zeta r_0-a )$.
\end{proof}

\begin{lem}\label{lem-2}
Fix a bounded subset $\scrA\subset\RR^{n-1}$ and $\zeta,\eta$ as in Lemma \ref{lem-1}. Then
\begin{equation}
\begin{split}
\int_{G/\Gamma} \#(  g\wbar \cap \scrZ(\zeta r -\eta,\infty,\scrA)) \,d\mu(g) 
& = \frac{\vol_{\H^n} \scrZ(\zeta r -\eta,\infty,\scrA)}{\#\Gamma_w \vol_{\HH^n}(\Gamma\bs\HH^n)} \\
& =  \frac{\e^{-(n-1) (\zeta r-\eta)}}{(n-1) \#\Gamma_w \vol_{\HH^n}(\Gamma\bs\HH^n)} .
\end{split}
\end{equation}
\end{lem}

\begin{proof}
This follows from \eqref{thesame}.
\end{proof}

\begin{lem}\label{lem-3}
Fix a bounded subset $\scrA\subset\RR^{n-1}$ and $\zeta,\eta$ as in Lemma \ref{lem-1}. Let $\lambda$ be a probability measure on $\TT^{n-1}$ with bounded density. Then there exists a constant $C$ such that for all $r\geq 0$
\begin{equation}
\sup_{t\geq 0} \int_{\TT^{n-1}} \#( \Phi^t n(\vecx) \wbar \cap \scrZ(\zeta r -\eta,\infty,\scrA)) \,d\lambda(\vecx) \leq C \e^{-(n-1) \zeta r} .
\end{equation}
\end{lem}

\begin{proof}
By increasing $C$, we may assume without loss of generality that $\lambda=\vol_{\TT^{n-1}}$. Then 
\begin{equation}
\begin{split}
\int_{\TT^{n-1}} & \#( \Phi^t n(\vecx) \wbar \cap \scrZ(\zeta r -\eta,\infty,\scrA)) \,d\vecx \\
& = \int_{\TT^{n-1}} \#( n(\vecx) \wbar \cap \scrZ(\zeta r -\eta-t,\infty,\e^{-t}\scrA)) \,d\vecx \\
& = \vol_{\RR^{n-1}}(\e^{-t}\scrA) \; \#\{ \gamma\in \Gamma_\infty\backslash\Gamma/\Gamma_w , \; \Im(\gamma w)\geq \e^{-t+\zeta r -\eta} \} .
\end{split}
\end{equation}
By the asymptotics \eqref{Nasy}, we find a constant $C'$ such that 
\begin{equation}
 \#\{ \gamma\in \Gamma_\infty\backslash\Gamma/\Gamma_w , \; \Im(\gamma w)\geq \e^{-t+\zeta r -\eta} \} \leq C' \max( 1 , \e^{(n-1)(t-\zeta r)} ).
\end{equation}
\end{proof}

\begin{proof}[Proof of Theorem \ref{th:moment-cusp}]
As remarked above, in the case $\slim<\infty$ the number $\scrN_{t,s}^\infty (\scrA_j,\vecx;\wbar)$ has a uniform upper bound, and $E_{\slim}(r,\scrA;\wbar)=0$ for $|r|:=\max_j r_j$ sufficiently large. The statement therefore follows directly from the convergence in distribution, Theorem \ref{th:distributionA}. Assume now $\slim=\infty$. 

For $\widetilde\scrA=\cup_j\scrA_j$ we have
\begin{equation}
\begin{split}
\sum_{|r|\geq R} E_{\slim}(r,\scrA;\wbar) & \leq \sum_{r'= R}^\infty E_{\slim}(r',\widetilde\scrA;\wbar) \\
& = \mu(\{ g\in G/\Gamma : \#( g\wbar \cap \scrZ(0,\infty,\widetilde\scrA)) \geq R \}) \\
& \leq \mu(\{ g\in G/\Gamma : \#( g\wbar \cap \scrZ(\zeta R -\eta,\infty,\widetilde\scrA)) \geq 1\}) \\
& \leq \int_{G/\Gamma}  \#( g\wbar \cap \scrZ(\zeta R -\eta,\infty,\widetilde\scrA)) \, d\mu(g).
\end{split}
\end{equation}
by Lemma \ref{lem-1} and Chebyshev's inequality. Using Lemma \ref{lem-2} on the last expression yields
\begin{equation}
\sum_{|r|\geq R} E_{\slim}(r,\scrA;\wbar) \leq C_1 \e^{-(n-1)\zeta R}.
\end{equation}
for an explicit constant $C_1$. This proves part (i). 

By Theorem \ref{th:distributionA}, we have
\begin{multline}\label{lasty}
\lim_{t\to\infty} \int_{\TT^{n-1}} 
\prod_{j=1}^m \ind{\scrN_{t,s}^\infty (\scrA_j,\vecx;\wbar)<R} \exp\big(\tau_j \scrN_{t,s}^\infty (\scrA_j,\vecx;\wbar) \big) 
d\lambda(\vecx) \\
=\sum_{r_1,\ldots,r_m=0}^{R-1} \exp\bigg(\sum_{j=1}^m \tau_j r_j\bigg) E_{\slim}(r,\scrA;\wbar) .
\end{multline}
To establish part (ii), what therefore remains to be shown is 
\begin{equation}
\lim_{R\to\infty}\limsup_{t\to\infty} \bigg| \int_{\TT^{n-1}} 
\prod_{j=1}^m \ind{\max_j \scrN_{t,s}^\infty (\scrA_j,\vecx;\wbar)\geq R} \exp\big(\tau_j \scrN_{t,s}^\infty (\scrA_j,\vecx;\wbar) \big) 
d\lambda(\vecx) \bigg| = 0 .
\end{equation}
Now,
\begin{multline}
\bigg| \int_{\TT^{n-1}} 
\prod_{j=1}^m \ind{\max_j \scrN_{t,s}^\infty (\scrA_j,\vecx;\wbar)\geq R} \exp\big(\tau_j \scrN_{t,s}^\infty (\scrA_j,\vecx;\wbar) \big) d
\lambda(\vecx) \bigg| \\
\leq 
\int_{\TT^{n-1}} 
\ind{\scrN_{t,s}^\infty (\widetilde\scrA,\vecx;\wbar)\geq R} \exp\big(\widetilde\tau \scrN_{t,s}^\infty (\widetilde\scrA,\vecx;\wbar) \big)  d
\lambda(\vecx)
\end{multline}
where $\widetilde\scrA=\cup_j\scrA_j$ and $\widetilde\tau=\sum_j \Re_+\tau_j$. We have
\begin{equation}
\begin{split}
\int_{\TT^{n-1}} &
\ind{\scrN_{t,s}^\infty (\widetilde\scrA,\vecx;\wbar)\geq R} \exp\big(\widetilde\tau \scrN_{t,s}^\infty (\widetilde\scrA,\vecx;\wbar) \big)  d
\lambda(\vecx) \\
& = \sum_{r=R}^\infty \e^{\widetilde\tau r}  \int_{\TT^{n-1}} 
\ind{\scrN_{t,s}^\infty (\widetilde\scrA,\vecx;\wbar)=r}  d
\lambda(\vecx) \\
& \leq \sum_{r=R}^\infty \e^{\widetilde\tau r}  \int_{\TT^{n-1}} 
\ind{\scrN_{t,s}^\infty (\widetilde\scrA,\vecx;\wbar)\geq r}  d
\lambda(\vecx) .
\end{split}
\end{equation}
Using Lemma \ref{lem-1}, the Chebyshev inequality and Lemma \ref{lem-3}, we see that the last integral is bounded by 
\begin{equation}
\int_{\TT^{n-1}} 
\ind{\scrN_{t,s}^\infty (\widetilde\scrA,\vecx;\wbar)\geq r}  d
\lambda(\vecx) \leq C \e^{-(n-1) \zeta r} ,
\end{equation}
uniformly in $t\geq 0$. This proves that, for $\widetilde\tau<(n-1) \zeta$, 
\begin{equation}
\lim_{R\to 0} \sum_{r=R}^\infty \e^{\widetilde\tau r}  \int_{\TT^{n-1}} 
\ind{\scrN_{t,s}^\infty (\widetilde\scrA,\vecx;\wbar)\geq r}  d\lambda(\vecx) 
= 0
\end{equation}
uniformly in $t$, which yields \eqref{lasty}.
\end{proof}

\section{Equidistribution of large spheres}\label{sec:Sphere}

We now consider the spherical analogue of horospherical averages by replacing $n(\vecx)$ with suitable rotation matrices $R(\vecx)\in K:=\PSU(C_{n-2})$, where $\vecx$ ranges over some open subset of $\scrU\subset\RR^{n-1}$. In geometric terms, it is natural to identify $\scrU$ either with (a subset of) the unit sphere centered at $\j$ via the map $\vecx\mapsto R(\vecx)^{-1} \e^{-1} \j$ (where $\e^{-1}\j$ is the ``south  pole'' of $\SS^{n-1}$), or with (a subset of) $\partial\HH^{n}$ via $\vecx\mapsto R(\vecx)^{-1} 0$. The sphere $\SS^{n-1}$ and the boundary $\d\HH^{n}$ are diffeomorphic, so that smoothness assumptions are mutually equivalent.

One important example for $R(\vecx)$ is the matrix
\begin{equation}\label{KANDEDEF}
E(\vecx):=\exp \begin{pmatrix} 0 & \vecx \\ -\vecx' & 0 \end{pmatrix} .
\end{equation}
Note that, we have for $\vecx\neq 0$ and $\widehat\vecx:=|\vecx|^{-1} \vecx$,
\begin{equation}
\begin{split}
E(\vecx) & = \begin{pmatrix} \widehat\vecx & 0 \\ 0 &  1 \end{pmatrix} 
\; E(|\vecx|) \;
\begin{pmatrix} \widehat\vecx & 0 \\ 0 &  1 \end{pmatrix}^{-1} \\
& = \begin{pmatrix} \widehat\vecx & 0 \\ 0 & 1 \end{pmatrix} 
\begin{pmatrix} \cos  |\vecx| & \sin |\vecx| \\ -\sin |\vecx| & \cos|\vecx| \end{pmatrix} 
\begin{pmatrix} \widehat\vecx & 0 \\ 0 &  1 \end{pmatrix}^{-1} \\
& =  \begin{pmatrix} \cos  |\vecx| & \widehat\vecx \sin |\vecx| \\ -\widehat\vecx' \sin |\vecx| & \cos|\vecx| \end{pmatrix}  ,
\end{split} 
\end{equation}
which shows that $E(\vecx)\in \SU(2,C_{n-2})$ (cf.\ \eqref{eq:SUdef}). Now
\begin{equation}
E^{-1}(\vecx) = \begin{pmatrix} \cos  |\vecx| & -\widehat\vecx \sin |\vecx| \\ \widehat\vecx' \sin |\vecx| & \cos|\vecx| \end{pmatrix}  
\end{equation}
and hence $E^{-1}(\vecx)0 = -\widehat\vecx \tan|\vecx|$. The map $\vecx\mapsto E(\vecx)^{-1} 0$ has thus nonsingular differential when $|\vecx| <\pi/2$
%\notin\{\pi,2\pi,3\pi,\ldots\}$, and 
(it is of course smooth everywhere in dimension $n=2$).

\begin{theorem}\label{th:spherical}
Let $\scrU\subset \R^{n-1}$ be a nonempty open subset and
let $R: \scrU\to K$ be a smooth map such that the map
$\scrU\to\partial\HH^{n}$, $\vecx\mapsto R^{-1}(\vecx) 0$, 
has nonsingular differential at Lebesgue-almost all $\vecx\in \scrU$.
Let $\lambda$ be a Borel probability measure on $\scrU$, 
absolutely continuous with respect to the Lebesgue measure.
Then, for any bounded continuous function $f:\scrU\times G/\Gamma\to \R$
and any family of uniformly bounded continuous functions 
$f_t: \scrU\times G/\Gamma\to \RR$ such that $f_t\to f$ as $t\to\infty$,
uniformly on compacta, and for every $g\in G$, we have
\begin{equation}\label{eq:spherical}
 \lim_{t\to\infty} \int_{\scrU} f_t\left(\vecx, \Phi^t R(\vecx) g  \right) d\lambda(\vecx) = \int_{\scrU\times G/\Gamma} f(\vecx,h)\,d\lambda(\vecx)\,d\mu(h).
\end{equation}
\end{theorem}

\begin{proof}
We proceed as in the proof of Corollary 5.4 in \cite{marklof_strombergsson_free_path_length_2010}. Let $\vecx_0\in \scrU$ be a point where the map $\vecx\mapsto R^{-1}(\vecx) 0$ has nonsingular differential. First we show that there exists an open set $\scrU_0\subset \scrU$ containing $\vecx_0$ such that \eqref{eq:spherical} holds with $\scrU$ replaced by $\scrU_0$ or any Borel subset thereof. 
Let
\beq\label{eq:R}
R(\vecx)   = \begin{pmatrix} \veca(\vecx) & \vecb(\vecx)\\ -\vecb'(\vecx) & \veca'(\vecx)\end{pmatrix}\in  K.
\eeq

{\em Case 1.} Suppose first that $\veca(\vecx_0)\ne 0$. Then $R(\vecx)^{-1}0 = -\veca^{-1}(\vecx) \vecb(\vecx) \in V_{n-2}$, and
\beq
R(\vecx) = \begin{pmatrix}\veca & 0\\ -\vecb' & \vecb' \veca^{-1}\vecb +\veca'\end{pmatrix} \begin{pmatrix}1 & \veca^{-1}\vecb\\0 & 1\end{pmatrix}.
\eeq
By assumption, the map $\vecx\mapsto \tilde \vecx(\vecx):=\veca^{-1}(\vecx) \vecb(\vecx)$ has nonsingular differential at $\vecx=\vecx_0$, so there exists an open set $\scrV$ containing $\vecx_0$ with $\overline{\scrV}\subset \scrU$ such the map $\vecx\mapsto \tilde \vecx$ is a diffeomorphism on $\scrV$. We call its image $\tilde \scrV$. Thus, 
\begin{align}
 \Phi^t R(\vecx) & = \Phi^t \begin{pmatrix}\veca & 0\\ -\vecb' & \vecb' \tilde \vecx +\veca'\end{pmatrix} \begin{pmatrix}1 & \tilde \vecx\\0 & 1\end{pmatrix}\\
 & = \begin{pmatrix}\veca & 0\\ -\vecb'\e^{-t} & \vecb' \tilde \vecx  +\veca'\end{pmatrix} \Phi^t \begin{pmatrix}1 & \tilde \vecx\\0 & 1\end{pmatrix}.
\end{align}

Fix $\scrU_0$ to be an open neighborhood of $\vecx_0$ such that $\overline{\scrU_0}\subset \scrV$, and let $B$ be a Borel subset of $\scrU_0$. Denote by $\tilde B$ and $\tilde \scrU_0$ the images of $B$ and $\scrU_0$ under $\vecx\mapsto \tilde \vecx$. We have $\tilde B\subset \tilde \scrU_0\subset \tilde \scrV$. Let us assume $\lambda( B) >0$, and let $\tilde \lambda$ be the measure on $\R^{n-1}$ which is the pushforward of $\frac1{\lambda (B)} \lambda\vert_B$ under the map $\vecx\mapsto \tilde \vecx$. Then $\tilde \lambda$ is a Borel probability measure with compact support and is absolutely continuous with respect to the Lebesgue measure. Finally, let $u$ be a continuous function satisfying $\chi_{\tilde \scrU_0} \le u \le \chi_{\tilde \scrV}.$ 

With $f$ and $f_t$ as in the statement, define continuous functions $\tilde f_t, \tilde f : \R^{n-1}\times G/\Gamma \to \R$ 
\begin{align}
 \tilde f_t(\tilde \vecx, h) &  = u(\tilde \vecx) f_t\left( \vecx,\begin{pmatrix}\veca & 0\\ -\vecb'\e^{-t} & \vecb' \tilde \vecx  +\veca'\end{pmatrix} h\right)& \text{ if }\tilde \vecx&\in \tilde \scrV\\
  \tilde f(\tilde \vecx, h) &  = u(\tilde \vecx) f\left( \vecx,\begin{pmatrix}\veca & 0\\ 0 & \vecb' \tilde \vecx  +\veca'\end{pmatrix} h\right)& \text{ if }\tilde \vecx&\in \tilde \scrV\\
  \tilde f_t(\tilde \vecx, h) & = \tilde f(\tilde \vecx,h)=0 &\text{ if }\tilde \vecx&\not\in \tilde \scrV.
\end{align}
We of course have that $\tilde f_t(\tilde \vecx, h)\to \tilde f(\tilde \vecx, h)$ as $t\to\infty$ uniformly on compact sets. Now we invoke Theorem \ref{th:horospherical} for $\tilde \lambda$, $\tilde f_t$, and $\tilde f$ to get 
\beq
\lim_{t\to\infty} \int_{\R^{n-1}} \tilde f_t(\tilde \vecx, \Phi^t n(\tilde \vecx)g) d\tilde \lambda(\tilde \vecx) = \int_{\R^{n-1}\times G/\Gamma} \tilde f(\tilde \vecx, h) d\mu(h) d\tilde \lambda(\tilde \vecx).
\eeq
Unwrapping the definition of $\tilde \lambda$ and using left invariance of $\mu$ we confirm that \eqref{eq:spherical} holds when $\scrU$ is replaced by any Borel subset $B$ of $\scrU_0$, provided $\lambda(B)>0$ (otherwise the claim is trivially true). 

{\em Case 2.} Suppose now $\veca(\vecx_0)=0$.
Note that in the definition of $R$ in \eqref{eq:R}, if $\veca(\vecx_0)=0$, then  $\vecb(\vecx_0)\ne 0$ and is hence invertible. With this in mind we write 
\beq 
R(\vecx) = R_0(\vecx) \begin{pmatrix} 0 & 1\\ -1 & 0\end{pmatrix}, \qquad R_0(\vecx):=\begin{pmatrix} \veca_0(\vecx) & \vecb_0(\vecx)\\ -\vecb'_0(\vecx) & \veca'_0(\vecx)\end{pmatrix} . 
\eeq
Then, $\veca=-\vecb_0$ and $\vecb=\veca_0\ne 0$, and the map $\vecx\mapsto R_0^{-1}(\vecx) 0 = \begin{pmatrix} 0 & 1\\ -1 & 0\end{pmatrix} R^{-1}(\vecx) 0$ has nonsingular differential at $\vecx=\vecx_0$. We conclude the argument as in Case 1 with $g$ replaced by $\begin{pmatrix} 0 & 1\\ -1 & 0\end{pmatrix}g$. 

The proof is now completed by a simple covering argument; see the end of the proof of Corollary 5.4 in \cite{marklof_strombergsson_free_path_length_2010} for details.
\end{proof}

As in the case for horospherical averages, we can extend Theorem \ref{th:spherical} to sequences of characteristic functions.

\begin{cor}\label{charThm-spherical}
Under the assumptions of Theorem \ref{th:spherical}, for any family of subsets $\scrE_t\subset\scrU\times G/\Gamma$ and any $g\in G/\Gamma$, we have
\begin{equation}\label{inf-s}
	\liminf_{t\to\infty} \int_{\scrU} \chi_{\scrE_t}(\vecx,\Phi^t R(\vecx) g) \,d\lambda(\vecx) \geq \int_{\lim(\inf \scrE_t)^\circ} d\lambda\,d\mu ,
\end{equation}
and
\begin{equation}\label{sup-s}
	\limsup_{t\to\infty} \int_{\scrU} \chi_{\scrE_t}(\vecx,\Phi^t R(\vecx) g)\,  d\lambda(\vecx) \leq \int_{\lim\overline{\sup \scrE_t}} d\lambda\,d\mu.
\end{equation}
If furthermore $\lambda\times\mu $ gives zero measure to the set $\lim\overline{\sup \scrE_t}\setminus\lim(\inf \scrE_t)^\circ$, then 
\begin{equation}\label{lim-s}
	\lim_{t\to\infty} \int_{\scrU}\chi_{\scrE_t}(\vecx,\Phi^t R(\vecx) g)\,  d\lambda(\vecx) = \int_{\limsup \scrE_t} d\lambda\,d\mu .
\end{equation}
\end{cor}

\section{Projection statistics for non-cuspidal observer}\label{sec:Projection-sph}

Let us now return to the study of the fine-scale statistics of the multiset of directions of lattice points $\scrP_{t,s}(g\wbar)$ as seen from an observer at the origin $z=\j$. Measuring correlations on the sphere of directions is a little more awkward than on a torus/horosphere, since rotations generally do not commute. We use the matrix $E(\vecx)$ to obtain a coordinate chart of a small neighborhood of the south pole of the sphere $\SS^{n-1}$ via the map $\vecx\to E(\vecx)^{-1} \e^{-1} \j$. We then define a shrinking test set in that neighborhood by
\begin{equation}
\scrB_{t,s}(\scrA,0) := \{ E(\vecx)^{-1} \e^{-1} \j : \vecx\in\rho_{t,s} \scrA \}
\end{equation}
where $\scrA\subset\RR^{n-1}$ is a fixed bounded set and the scaling factor $\rho_{t,s}>0$ is chosen so that 
\begin{equation}
\omega( \scrB_{t,s}(\scrA,0) )= \frac{\Omega_n\vol_{\RR^{n-1}}\scrA}{\#\scrP_{t,s}(g\wbar)}.
\end{equation}
(We will see below that $\rho_{t,s} \sim \vartheta^{-1/(n-1)}\, \e^{-t}$ for $t$ large.)
To rotate this set randomly, we use the map $\vecx\mapsto R(\vecx)$ of a open subset $\scrU\subset\RR^{n-1}$ defined in the previous section, and set
\begin{equation}\label{Ats-sph}
\scrB_{t,s}(\scrA,\vecx) := R(\vecx)^{-1} \scrB_{t,s}(\scrA,0) .
\end{equation}
A key observation will be that the limit distribution of the random variable
\begin{equation}\label{Nts-sph}
\scrN_{t,s}(\scrA,\vecx;g\wbar) := \#(\scrP_{t,s}(g\wbar) \cap \scrB_{t,s} (\scrA,\vecx))
\end{equation}
will be independent of the choice of $R$.
As before, the scaling of the test set ensures that, for any probability measure $\lambda$ with continuous density, 
\begin{equation}\label{expect22-sph}
\lim_{t\to\infty} \int_{\scrU} \scrN_{t,s}(\scrA,\vecx;g\wbar)\, d\lambda(\vecx) = \vol_{\RR^{n-1}} \scrA .
\end{equation}
This formula also follows from the convergence of moments (see Section \ref{sec:Moments-sph}) and the explicit formula \eqref{five-six} for the first moment of the limit distribution, under the weaker assumption that $\lambda$ has bounded density.

As in the case of a  cuspidal observer, we consider the joint distribution with respect to several test sets $\scrA_1,\ldots,\scrA_m$:

\begin{thm}\label{th:distributionB}
Let $\scrU\subset \R^{n-1}$ be a nonempty open subset and
let $R: \scrU\to K$ be a smooth map such that the map
$\scrU\to\partial\HH^{n}$, $\vecx\mapsto R^{-1}(\vecx) 0$,
has nonsingular differential at Lebesgue-almost all $\vecx\in \scrU$.
Let $\lambda$ be a Borel probability measure on $\scrU$ absolutely continuous with respect to the Lebesgue measure.
Then, for every $g\in G$, $\slim\in(0,\infty]$, $r=(r_1,\ldots,r_m)\in\ZZ_{\geq 0}^m$ and $\scrA=\scrA_1\times\cdots\times\scrA_m$ with $\scrA_j\subset\RR^{n-1}$ bounded with boundary of Lebesgue measure zero,
\begin{equation}
\lim_{t\to\infty} \lambda( \{ \vecx\in\scrU : \scrN_{t,s} (\scrA_j,\vecx;g\wbar) = r_j \;\forall j \}) = E_{\slim}(r,\scrA;\wbar)
\end{equation}
where the limit distribution $E_{\slim}(r,\scrA;\wbar)$ is the same as in the case of a cuspidal observer in Theorem \ref{th:distributionA}. In particular, the limit is independent of $g$, $R$, $\lambda$, and $\scrU$.
\end{thm}

The proof of Theorem \ref{th:distributionB} is almost identical to that of Theorem \ref{th:distributionA}, with horospherical averages replaced by spherical averages (Corollary \ref{charThm-spherical}), and Lemma \ref{sat-lem} replaced by the following.

\begin{lemma}\label{sat-lem-sph}
Under the hypotheses of Theorem \ref{th:distributionB}, given $\epsilon>0$ there exist $t_0<\infty$ and bounded subsets $\scrA_j^-\subset\scrA_j^+\subset\RR^{n-1}$ with boundary of measure zero, such that 
\begin{equation}
\vol_{\RR^{n-1}}(\scrA_j^+\setminus\scrA_j^-)<\epsilon
\end{equation}
and, for all $t\geq t_0$,
\begin{align}
\#( \Phi^t R(\vecx) g\wbar \cap \scrZ(\epsilon,\slim^-,\scrA_j^-) )\leq \scrN_{t,s} (\scrA_j,\vecx;g\wbar)  
\leq  \#( \Phi^t R(\vecx) g\wbar \cap \scrZ(-\epsilon,\slim+\epsilon,\scrA_j^+) ) 
\end{align}
with
\begin{equation}
\slim^- =
\begin{cases} \slim-\epsilon & (\slim<\infty) \\
\epsilon^{-1} & (\slim = \infty) .
\end{cases}
\end{equation}
\end{lemma}

\begin{proof}
For $\scrB\subset\SS^{n-1}$ and $-\infty\leq a<b< \infty$ define the cone
\begin{equation}
\scrC(a,b,\scrB) := 
\{ z\in\HH^n\setminus\{\j\} : \varphi(z)\in\scrB,\; a<d(z)\leq b \}  .
\end{equation}
Its volume satisfies
\begin{equation}\label{volli}
\vol_{\HH^n} \scrC(a,b,\scrB) = \frac{\omega(\scrB)}{\Omega_n}\; \vol_{\HH^n} \scrC(a,b,\SS^{n-1}) ,
\end{equation}
where $\scrC(a,b,\SS^{n-1})$ is the spherical shell with inner resp.\ outer radius $a$ and $b$.
Now
\begin{equation}
\begin{split}
\scrN_{t,s} (\scrA_j,\vecx;g\wbar) & = \#( g\wbar \cap \scrC(t-s,t,\scrB_{t,s}(\scrA_j,\vecx))) \\
& = \#( \Phi^t R(\vecx) g\wbar \cap \Phi^t \scrC(t-s,t,\scrB_{t,s}(\scrA_j,0))) ,
\end{split}
\end{equation}
where
\begin{equation}
\Phi^t  \scrC(t-s,t,\scrB_{t,s}(\scrA_j,0)) = \bigcup_{0\leq r<\min\{t,s\}} \{ \Phi^t  E(\vecx)^{-1} \e^{r-t} \j : \vecx\in \rho_{t,s} \scrA_j \} .
\end{equation}
Note that $\rho_{t,s}$ may depend on $\scrA_j$.
The volume of this cone is in view of \eqref{volli}, for $t$ large,
\begin{equation}
\begin{split}
\vol_{\HH^n}  \scrC(t-s,t,\scrB_{t,s}(\scrA_j,0)) 
& \sim \e^{(n-1)t} \,\frac{1-\e^{-(n-1)s}}{n-1}\,\omega(\scrB_{t,s}(\scrA_j,0)) \\
& = \e^{(n-1)t} \,\frac{1-\e^{-(n-1)s}}{n-1}\,\frac{\Omega_n\vol_{\RR^{n-1}}\scrA}{\#\scrP_{t,s}(g\wbar)} \\
& \sim \#\Gamma_w \vol_{\HH^n}(\Gamma\bs\HH^n)\, \vol_{\RR^{n-1}}\scrA,
\end{split}
\end{equation}
the same volume as the cuspidal cone $\scrZ(\slim,\scrA)$, \eqref{eq:volume_cylinder}.
We have, 
\begin{equation}
\begin{split}
\Phi^t  E(\vecx)^{-1}  & = \Phi^t \bigg( 1 + \begin{pmatrix} 0 & -\vecx \\ \vecx' & 0 \end{pmatrix} + O(\rho_{t,s}^2) \bigg) \\
& = \bigg( 1 + \begin{pmatrix} 0 & -\e^t \vecx \\ \e^{-t} \vecx' & 0 \end{pmatrix} + O(\e^t \rho_{t,s}^2) \bigg) \Phi^t \\
& = \bigg( n(-\e^t \vecx) + O(\e^{-t}\rho_{t,s}) + O(\e^t \rho_{t,s}^2) \bigg) \Phi^t 
\end{split}
\end{equation}
and so
\begin{equation}
\Phi^t  E(\vecx)^{-1} \e^{r-t} \j = -\e^t \vecx + \e^r \j +\text{ lower order terms.} 
\end{equation}
This shows that, for every fixed $r$, the set 
\begin{equation}
\{ \Phi^t  E(\vecx)^{-1} \e^{r-t} \j : \vecx\in\rho_{t,s} \scrA_j \}
\end{equation}
is close to
\begin{equation}
\{ -\vecx + \e^r \j : \vecx\in \e^t \rho_{t,s} \scrA_j \} .
\end{equation}
We conclude that, for $t$ large, $\Phi^t  \scrC(t-s,t,\scrB_{t,s}(\scrA_j,0))$ approximates $\scrZ(\slim,\vartheta^{1/(n-1)} \e^t \rho_{t,s} \scrA_j) )$. Comparing the volumes of the two yields $\rho_{t,s} \sim \vartheta^{-1/(n-1)}\, \e^{-t}$.
\end{proof}

Theorem \ref{th:distribution} is now a corollary of Theorem \ref{th:distributionB}.

\begin{proof}[Proof of Theorem \ref{th:distribution}]
In Theorem \ref{th:distributionB}, choose $m=1$ and $\scrA\subset\RR^{n-1}$ a Euclidean open ball of volume $\sigma$. Then
\begin{equation}
\scrB_{t,s}(\scrA,0) = \{ E(\vecx)^{-1} \e^{-1} \j : \vecx\in \rho_{t,s} \scrA \} =\scrD_{t,s} (\sigma,\e^{-1}\j),
\end{equation}
which is the spherical disc with the required volume \eqref{eq:scaling}. Define  coordinate charts
\begin{equation}
\scrU \to \SS^{n-1}, \qquad \vecx \mapsto \vecv=R(\vecx)^{-1} \e^{-1} \j,
\end{equation}
for suitable $\scrU$ and $R(\vecx)$ with non-singular differential. (Take for instance $R(\vecx)=R_0 E(\vecx)$ which will parametrize a neighborhood of any point $\vecv_0=R_0^{-1} \e^{-1} \j\in\SS^{n-1}$.) This implies that $\scrU\to\partial\HH^{n}$, $\vecx\mapsto R^{-1}(\vecx) 0$ has non-singular differential, and we can apply Theorem \ref{th:distributionB} to prove \eqref{th:distribution-eq1}, with $\lambda$ in \eqref{th:distribution-eq1} restricted to each coordinate chart. The cuspidal cone in Theorem \ref{th:distribution} is defined by 
\begin{equation}\label{zyl0-def}
\scrZ_0(\slim,\sigma):=\scrZ(\slim,\scrA) ,
\end{equation}
with $\scrA$ a ball of volume $\sigma$ as above.
For $\sigma=0$ or $\slim=0$, we define $\scrZ_0(\slim,\sigma)$ as the empty set. The continuity in $\slim$ and $\sigma$, as well as \eqref{01law}, follow from the continuity stated in \eqref{conti}.
\end{proof}

\section{Convergence of moments for non-cuspidal observer}\label{sec:Moments-sph}

In analogy with Section \ref{sec:Moments}, we define the moment generating function for a non-cuspidal observer by
\begin{equation}
\GG_{t,s}(\tau_1,\ldots, \tau_m;\scrA) := \int_{\scrU} \exp\bigg(\sum_{j=1}^m \tau_j \scrN_{t,s}(\scrA_j,\vecx;g\wbar) \bigg) d\lambda(\vecx) ,
\end{equation}
where $\scrU$, $\lambda$ are as in the previous section.

\begin{thm}\label{th:moment-sph}
Let $\lambda$ be a probability measure on $\scrU$ with bounded density. 
Then there is a constant $c_0>0$ such that for $\Re_+\tau_1+\ldots + \Re_+\tau_m< c_0$, $\slim\in(0,\infty]$, 
\begin{equation}
\lim_{t\to\infty} \GG_{t,s}(\tau_1,\ldots, \tau_m;\scrA)  = \GG_{\slim}(\tau_1,\ldots, \tau_m;\scrA),
\end{equation}
with $\GG_{\slim}(\tau_1,\ldots, \tau_m;\scrA)$ as defined in \eqref{limit-G}.
\end{thm}

Theorem \ref{th:moment-sph} implies convergence of the mixed moments
\begin{equation}
\MM_{t,s}(\beta_1,\ldots, \beta_m;\scrA) := \int_{\scrU} \prod_{j=1}^m\left(\scrN_{t,s}(\scrA_j,\vecx;g\wbar)\right)^{\beta_j} d\lambda(\vecx) 
\end{equation}
for all $\beta_j\in\RR_{\geq 0}$. 

\begin{cor}\label{cor:moment-sph}
Let $\lambda$ be a probability measure on $\scrU$ with bounded density. Then, for all $\beta_1,\ldots,\beta_m\in\RR_{\geq 0}$, $\slim\in(0,\infty]$,
\begin{equation}
\lim_{t\to\infty} \MM_{t,s}(\beta_1,\ldots, \beta_m;\scrA)   = \MM_{\slim}(\beta_1,\ldots, \beta_m;\scrA)
\end{equation}
with $\MM_{\slim}(\beta_1,\ldots, \beta_m;\scrA)$ as defined in \eqref{def:moment}.
\end{cor}

The proof of Theorem \ref{th:moment-sph} is completely analogous to that of Theorem \ref{th:moment-cusp}. It is again based on Lemma \ref{lem-2} and the following two lemmas, which substitute Lemmas \ref{lem-1} and \ref{lem-3}, respectively.

\begin{lem}\label{lem-1-sph}
Fix $a\in\RR$ and a bounded subset $\scrA\subset\RR^{n-1}$. The there exists positive constants $\zeta,\eta,t_0$ so that for all $g\in G$, $r\in\NN$, $t\geq t_0$
\begin{equation}\label{ii1-sph}
\big[ \#( g\wbar \cap \scrC(0,t,\scrB_{t,\infty}(\scrA,0))) \geq r \big] \quad \Rightarrow \quad \big[\#( g\wbar \cap \scrC(0,t-\zeta r + \eta,\scrB_{t,\infty}(\scrA,0))) \geq 1\big].
\end{equation}
\end{lem}

\begin{proof}
This is analogous to the proof of Lemma \ref{lem-1}.
\end{proof}

\begin{lem}\label{lem-3-sph}
Fix a bounded subset $\scrA\subset\RR^{n-1}$ and $\zeta,\eta$ as in Lemma \ref{lem-1}. Let $\lambda$ be a probability measure on $\scrU$ with bounded density. Then there exists a constant $C$ such that for all $r\geq 0$
\begin{equation}
\sup_{t\geq 0} \int_{\scrU} \#( \Phi^t R(\vecx) g\wbar \cap \scrC(0,t-\zeta r + \eta,\scrB_{t,\infty}(\scrA,0))) \,d\lambda(\vecx) \leq C \e^{-(n-1) \zeta r} .
\end{equation}
\end{lem}

\begin{proof}
To obtain an upper bound, we may replace $\scrB_{t,\infty}(\scrA,0)$ by a sufficiently large ball $\scrD_t\subset\SS^{n-1}$ so that $\scrB_{t,\infty}(\scrA,0)\subset\scrD_t$ and $\omega (\scrD_t )= \sigma_0 \e^{-(n-1)t}$ for all $t\geq 0$ and some constant $\sigma_0$. Since $\lambda$ has bounded density and $\scrD_t$ is rotation invariant, there is a constant $C_2$ such that
\begin{multline}
\int_{\scrU} \#( \Phi^t R(\vecx) g\wbar \cap \scrC(0,t-\zeta r + \eta,\scrB_{t,\infty}(\scrA,0))) \,d\lambda(\vecx) \\
\leq 
C_2 \int_K \#( \Phi^t k g\wbar \cap \scrC(0,t-\zeta r + \eta,\scrD_t)) \,dm(k) ,
\end{multline}
where $m$ is the Haar probability measure on $K$.
We have
\begin{multline}
\int_K \#( \Phi^t k g\wbar \cap \scrC(0,t-\zeta r + \eta,\scrD_t)) \,dm(k) \\
= \sigma_0 \e^{-(n-1)t} \; \#\{ \gamma\in \Gamma/\Gamma_w , \; 0< d(\gamma g w)\leq \e^{t-\zeta r + \eta} \} .
\end{multline}
Finally, in view of the asymptotics \eqref{ball-count}, there is a constant $C_3$ such that, for all $t\geq 0$,
\begin{equation}
  \#\{ \gamma\in \Gamma/\Gamma_w , \; d(\gamma g w)\leq \e^{t-\zeta r + \eta} \} \leq C_3 \max( 1 , \e^{(n-1)(t-\zeta r)} ) .
\end{equation}
This completes the proof of the lemma.
\end{proof}

\newcommand{\originalvariable}{\xi}
\newcommand{\finalvariable}{\alpha}

\begin{appendix}
 \section{Two-point correlation functions}\label{sec:correlations}

This appendix derives the explicit formula for the 2-point function in dimension $n=2$ and $s=\infty$, which was first calculated in the work of Boca et al.\ \cite{boca_pair_2013}  and Kelmer and Kontorovich \cite{kelmer_pair_2013}. As explained in Remark \ref{rem:2pt}, the convergence of the 2-point function follows from the convergence of the second mixed moment by a standard argument  (cf.~\cite[Appendix~1]{EMV_directions_2013}), and the limit 2-point function $R_2(\originalvariable )$ (defined as in  \cite[Eq.~(1.7)]{kelmer_pair_2013}) is related to the second mixed moment by
\begin{align}\label{eq:fullinterval}
R_2(\originalvariable ) &= \lim_{\epsilon\to 0} \frac1{4\epsilon} \,\big[ \MM_\infty\big(1,1;(-\originalvariable ,\originalvariable )\times (-\epsilon,\epsilon)\big) - \MM_\infty\big(1;(-\epsilon,\epsilon)\big) \big]
\end{align}
The term $\MM_\infty\big(1;(-\epsilon,\epsilon)\big)$ removes the diagonal contribution $\gamma_1=\gamma_2\in\Gamma/\Gamma_w$ in the sum defining the second moment. We divide by $4\eps$ rather than $2\eps$ in \eqref{eq:fullinterval} to count ordered pairs $(\gamma_1,\gamma_2)$ and $(\gamma_2,\gamma_1)$ in  \eqref{eq:mixmo} only once; this is consistent with the definition of $\scrN_Q(\xi )$ in \cite{kelmer_pair_2013}. 
(Note however that unlike \cite{boca_pair_2013,kelmer_pair_2013} we take $\gamma\in\Gamma/\Gamma_w$ in order to count each point in the orbit $\wbar=\Gamma w$ only once.)
In view of \eqref{eq:mixmo2},
\begin{equation}
\MM_\infty\big(1,1;(-\originalvariable,\originalvariable )\times (-\epsilon,\epsilon)\big) - \MM_\infty\big(1;(-\eps,\epsilon)\big)= 
\frac{1}{\#\Gamma_w} \sum_{\substack{\gamma\in\Gamma/\Gamma_w\\ \gamma\neq\Gamma_w}} F_{\gamma,\epsilon}(\vartheta^{-1}\originalvariable ) 
\end{equation}
with
\begin{equation}
F_{\gamma,\epsilon}(\finalvariable ) := \int_{G} \ind{g\gamma w\in\scrZ(\infty,(-\vartheta\finalvariable,\vartheta\finalvariable ))} \ind{g w\in\scrZ(\infty,(-\epsilon,\epsilon))} \,d\mu(g) .
\end{equation}
Note that for $n=2$ we have $\vartheta^{-1}= \#\Gamma_w\vol_{\H^2}(\Gamma\quot \H^2)$.
There is $h\in G$ such that $h w=\i$ and $h\gamma w=\e^{\ell}\i$ where $\ell:=d(w,\gamma w)$. Hence
\begin{equation}
F_{\gamma,\epsilon}(\finalvariable ) = \int_{G} \ind{g\e^{\ell}\i \in\scrZ(\infty,(-\vartheta\finalvariable,\vartheta\finalvariable ))} \ind{g\i\in\scrZ(\infty,(-\epsilon,\epsilon))} \,d\mu(g) .
\end{equation}
In dimension $n=2$, the Iwasawa decomposition \eqref{Iwasawa} for $G=\PSL(2,\R)$ reads
\begin{equation}
g= n(x) a(y) k(\theta), \qquad k(\theta)=\begin{pmatrix}\cos \theta & -\sin\theta\\ \sin\theta & \cos \theta\end{pmatrix}, 
\end{equation}
where $0\leq \theta<\pi$, and Haar measure \eqref{iwa}
\begin{equation}
d\mu(g) = \kappa\, \frac{dx\,dy\,d\theta}{\pi y^2} ,
\end{equation}
where $\kappa=\vol_{\H^2}(\Gamma\quot \H^2)^{-1}$.
With this,
\begin{equation}
F_{\gamma}(\finalvariable ) := \lim_{\epsilon\to 0} \frac{1}{4\epsilon} F_{\gamma,\epsilon}(\finalvariable ) = \frac{\kappa}{2\pi\vartheta}\int_0^\pi \int_1^{\infty}  \ind{a(y)k(\theta) \e^{\ell}\i \in\scrZ(\infty,(-\vartheta\finalvariable,\vartheta\finalvariable ))} \,\frac{dy}{y^2}\, d\theta ,\label{eq:paircorrgamma}
\end{equation}
where $\kappa/\vartheta=\#\Gamma_w$.
The indicator function restricts the domain of integration to
\begin{equation}
\frac{y}{\ch \ell-\sh \ell \cos2\theta}>1, \qquad -\finalvariable < \frac{y\sin2\theta\sh \ell}{\ch \ell-\sh \ell \cos2\theta}<\finalvariable.
\label{eq:alpha_def}
\end{equation}
We exploit the symmetry of the domain of integration by noticing that $\theta<\pi/2$ if and only if $\frac{y\sin2\theta\sh \ell}{\ch \ell-\sh \ell \cos2\theta}>0$, which allows us to rewrite the integral as the sum of two equal  integrals, that over $(0,\pi/2)$ and that over $(\pi/2,\pi)$. Therefore,
\begin{align}
F_{\gamma}(\finalvariable )=\frac{\kappa}{\pi\vartheta}\int_0^\pi \int_1^{\infty}  \ind{a(y)k(\theta) \e^{\ell}\i \in\scrZ(\infty,(0,\vartheta\finalvariable ))} \,\frac{dy}{y^2}\, d\theta ,\label{eq:paircorrgamma_positive}
\end{align}
so that the range of integration becomes
\beq\label{eq:conditions}
 \{y>1\} \cap\left\{ y> \ch \ell - \sh \ell\cos2\theta>0\right\}\cap \left\{y\sin2\theta\sh \ell <(\ch \ell-\sh \ell\cos2\theta)\finalvariable  \right\}\cap \{\theta<\pi/2\}. 
\eeq
We remark that every side of every inequality above is positive. We seek to compute the derivative with respect to $\alpha $ of
\beq\int_{\eqref{eq:conditions}}
\frac{dy}{y^2} d\theta
\label{eq:integral}\eeq
for given $\finalvariable $ and $\ell$. 

The first two inequalities bound $y$ from below, so we split the integral over $\theta$ into two parts, according to which condition dominates: 
\begin{multline}\label{eq:split_y}
\eqref{eq:integral} = \! \int_{\theta: 1> \ch \ell-\sh \ell\cos 2\theta} \limits \! \left[ \int_{y=1}^{\frac{\finalvariable (\ch \ell-\sh \ell\cos 2\theta)}{\sin 2\theta \sh \ell}} \frac{dy}{y^2} \right]_+ d\theta 
\\ + \! \int_{\theta: 1 < \ch \ell-\sh \ell\cos 2\theta} \limits \!\left[\int_{y= \ch \ell-\sh \ell\cos 2\theta}^{\frac{\finalvariable (\ch \ell-\sh \ell\cos 2\theta)}{\sin 2\theta \sh \ell}} \frac{dy}{y^2}\right]_+ d\theta.
\end{multline}
Here $[\cdot]_+ = \max(0,\cdot)$, and serves the purpose of excluding regions where limits are reversed (upper limit is smaller than the lower limit). 

The range of integration  in $\theta$ is equivalent to $\cos 2\theta>\th \frac \ell2$ for the first term of \eqref{eq:split_y}, and the inequality is reversed in the second term. Since $\theta\in [0,\pi / 2)$, there is a unique $\theta_0$ so that the range of integration is from $0$ to $\theta_0$ for the first integral and $\theta_0$ to $\pi /2$ for the second. 

Consider the first term in \eqref{eq:split_y}. It evaluates to 
\begin{align}\label{eq:first_integral}\int_{\theta=0}^{\theta_0} \left[ 1 - \frac{\sin 2\theta \sh \ell}{\finalvariable (\ch \ell-\sh \ell\cos 2\theta)} \right]_+ d\theta.
\end{align}
Observe that the antiderivative can be explicitly written as
\beq\label{eq:antider1}
 \int \left(1 - \frac{\sin 2\theta \sh \ell}{\finalvariable  (\ch \ell-\sh \ell\cos 2\theta)}\right) d\theta = \theta-\frac1{2\finalvariable }\log(\ch \ell-\sh \ell\cos2\theta)+C.
\eeq
It remains to establish the correct range of integration. When $\finalvariable >\sh \ell$, the integrand is always nonnegative, so the $+$-sign is superfluous, and evaluating \eqref{eq:antider1} at $\theta_0$ and $0$ we get for \eqref{eq:first_integral}
\beq
 \theta_0-\frac \ell{2\finalvariable }.
\eeq
Consider the case $\finalvariable < \sh \ell$ and introduce the auxiliary variable $\phi\in [0,\pi/2]$ with $\sin \phi=\frac{\finalvariable }{\sqrt{1+\finalvariable ^2}}$. The positivity condition amounts to $\sin(2\theta +\phi)\le \sin\phi \cth \ell$. When $\theta\in[0,\pi/2)$, this condition is satisfied outside the interval $(\theta_-,\theta_+)$ with 
\begin{align}
 \theta_+ & =\frac\pi2 -\frac\phi2 -\frac12 \arcsin(\sin\phi \cth \ell)\\
 \theta_- & =\frac12 \arcsin(\sin\phi\cth \ell)-\frac\phi2.
\end{align}
Further analysis shows that so long as $\finalvariable  >2\sh\frac \ell2,$ we have the inclusion $(\theta_-,\theta_+)\subset [0,\theta_0)$, so that the range of integration consists of two intervals, $[0,\theta_-)\cup (\theta_+, \theta_0)$. 
If $\finalvariable  <2\sh \frac \ell2$, then $\theta_+>\theta_0$, and we need to integrate over but one interval, $[0,\theta_-)$. Exact formulas follow immediately by substituting into the antiderivative formula \eqref{eq:antider1}. 

Now consider the second term of \eqref{eq:split_y}, which evaluates to 
\beq\label{eq:second_integral}
 \int_{\theta=\theta_0}^{\pi/2} \left[\frac1{\ch \ell-\sh \ell\cos2\theta} \left(1-\frac{\sin 2\theta \sh \ell}{\finalvariable }\right)\right]_+ d\theta.
\eeq
Positivity here is determined by the sign of the expression in parentheses. The antiderivative reads 
\beq\label{eq:antider2}
\arctg(e^\ell \tg\theta)-\frac1{2\finalvariable } \log(\ch \ell-\sh \ell\cos 2\theta)+C .
\eeq
If $\finalvariable  > \sh \ell$, the positivity condition holds for all $\theta$, and the integral \eqref{eq:second_integral} equals 
\[\frac\pi 2 - \frac \ell{2\finalvariable  }-\arctg(e^{\ell/2}).\]
Otherwise within the interval $[0,\pi/2)$, the integrand is nonzero on the complement of $[\tilde \theta_-,\tilde \theta_+)$ with 
\begin{align}
 \tilde \theta_- & = \frac12 \arcsin \frac{\finalvariable  }{\sh \ell}\\
 \tilde \theta_+ & = \frac\pi 2 -\frac12 \arcsin  \frac{\finalvariable  }{\sh \ell}.
\end{align}
When $2 \sh \frac \ell2<\finalvariable  $, $[\tilde \theta_-, \tilde \theta_+]\subset (\theta_0, \pi/2)$, and the range of integration consists of two intervals. When $\finalvariable < 2\sh \frac \ell2,$ the range of integration of \eqref{eq:antider2} consists of the single interval $(\tilde\theta_+,\pi/2).$ Again, the integral \eqref{eq:second_integral} is evaluated by substituting limits in  \eqref{eq:antider2}.

For $F_\gamma'(\finalvariable):= \frac{d}{d \finalvariable  } F_\gamma( \finalvariable)$, this leads to
\begin{equation}\label{finaldensity}
F_\gamma'(\finalvariable) 
=\frac{\#\Gamma_w}{\pi \finalvariable^2} \begin{cases}
 \ell & \finalvariable >\sh \ell\\
\ell + \log(1+\finalvariable ^2) - 2\log  (\ch \ell+\sqrt{\sh^2 \ell -\finalvariable ^2})& 2\sh\frac \ell2 < \finalvariable  \le \sh \ell\\
\ell-\log (\ch \ell+\sqrt{\sh^2 \ell -\finalvariable ^2}) & \finalvariable  \le 2\sh \frac \ell2.
\end{cases}
\end{equation}
Therefore we have for the 2-point correlation density
\begin{align}
g_2(\xi):=\frac{d R_2}{d \originalvariable}( \originalvariable ) = \frac1{\vartheta \#\Gamma_w} \sum_{\substack{\gamma\in\Gamma/\Gamma_w\\ \gamma\neq\Gamma_w}} F_\gamma'(\vartheta^{-1} \originalvariable)
= \vol_{\HH^2}(\Gamma\bs\HH^2) \sum_{\substack{\gamma\in\Gamma/\Gamma_w\\ \gamma\neq\Gamma_w}}   F_\gamma'(\vartheta^{-1} \originalvariable) ,
\end{align}
matching the formulas in \cite{boca_pair_2013, kelmer_pair_2013}, up to the extra factor of $\#\Gamma_w$ in the definition of $\vartheta$, which is due to counting $\gamma$ in $\Gamma/\Gamma_w$ rather than in $\Gamma$ as in \cite{boca_pair_2013, kelmer_pair_2013}.

\end{appendix}

\bibliographystyle{plain}
\bibliography{bibliography}

\end{document}